\documentclass[reqno,10pt]{amsart}
\usepackage{fullpage,graphicx,amsfonts,amssymb,amsmath,amsthm,amscd,url}

\theoremstyle{plain} 
\newtheorem{theorem}    [subsection]{Theorem}
\newtheorem{lemma}      [subsection]{Lemma}
\newtheorem{corollary}  [subsection]{Corollary}
\newtheorem{proposition}[subsection]{Proposition}

\newcounter{custom}

\newtheorem{conjecture_custom}[custom]{Conjecture}
\newtheorem{conjecture_disproven}[custom]{Disproven Conjecture}

\theoremstyle{definition}
\newtheorem{definition} [subsection]{Definition}

\theoremstyle{remark}
\newtheorem{remark}[subsection]{Remark}

 \newcommand{\chern}[1]{\mathrm{c}_{#1}}

\def\Sym{\operatorname{Sym}}
\def\SO{\operatorname{SO}}
\def\OO{\operatorname{O}}
\def\id{\operatorname{id}}
\def\GL{\operatorname{GL}}
\def\Mon{\operatorname{Mon}}

\def\NE{\operatorname{NE}}
\def\Res{\operatorname{Res}}

\renewcommand{\P}{\mathbb{P}}
\newcommand{\Q}{\mathbb{Q}}
\newcommand{\Z}{\mathbb{Z}}
\newcommand{\C}{\mathbb{C}}

\renewcommand{\O}{\mathcal{O}}

\newcommand{\triv}{\mathbf{1}}
\newcommand{\pt}{[\mathrm{pt}]}
\renewcommand{\phi}{\varphi}
\newcommand{\into}{\rightarrow}
\newcommand{\euler}{\mathbf{e}}
\newcommand{\K}[1]{$K3^{[#1]}$}

\newcommand{\real}[1]{\ldots}

 \title{Lagrangian $4$-Planes in Holomorphic Symplectic Varieties of \K{4}-type}
 \date{\today}
 \author{Benjamin Bakker}
 \address{B. Bakker:
 Courant Institute of Mathematical Sciences, New York University,
 251 Mercer St., New York, NY 10012
 }
 \email{bakker@cims.nyu.edu}
 \author{Andrei Jorza}
 \address{A. Jorza:
 University of Notre Dame, 275 Hurley, Notre Dame, IN 46556}
 \email{ajorza@nd.edu}

\begin{document}


 \begin{abstract}
We classify the cohomology classes of Lagrangian 4-planes $\P^4$ in a smooth manifold $X$ deformation equivalent to a Hilbert scheme of 4 points on a $K3$ surface, up to the monodromy action.   Classically, the cone of effective curves on a $K3$ surface $S$ is generated by nonegative classes $C$, for which $(C,C)\geq0$, and nodal classes $C$, for which $(C,C)=-2$; Hassett and Tschinkel conjecture that the cone of effective curves on a holomorphic symplectic variety $X$ is similarly controlled by ``nodal'' classes $C$ such that $(C,C)=-\gamma$, for $(\cdot,\cdot)$ now the Beauville-Bogomolov form, where $\gamma$ classifies the geometry of the extremal contraction associated to $C$.  In particular, they conjecture that for $X$ deformation equivalent to a Hilbert scheme of $n$ points on a $K3$ surface, the class $C=\ell$ of a line in a smooth Lagrangian $n$-plane $\P^n$ must satisfy $(\ell,\ell)=-\frac{n+3}{2}$.  We prove the conjecture for $n=4$ by computing the ring of monodromy invariants on $X$, and showing there is a unique monodromy orbit of Lagrangian $4$-planes.
 \end{abstract}

\maketitle


 \section{Introduction}
Let $X$ be an irreducible holomorphic symplectic variety; thus, $X$ is a smooth projective simply-connected variety whose space $H^0(\Omega^2_X)$ of global two-forms is generated by a nowhere degenerate form $\omega$.
$H^2(X,\Z)$ carries a deformation-invariant nondegenerate primitive
integral form $(\cdot,\cdot)$ called the Beauville-Bogomolov
form \cite{beauville}.  For $X=S$ a $K3$ surface $(\cdot,\cdot)$ is the
intersection form, while for $X=S^{[n]}$ a Hilbert scheme of $n>1$
points on $S$ we have the orthogonal decomposition \cite[$\mathsection8$]{beauville}
\begin{equation}H^2(S^{[n]},\Z)_{(\cdot,\cdot)}\cong
H^2(S,\Z)\oplus_{\perp}\Z\delta\label{decomp}\end{equation}
where the form on $H^2(S,\Z)$ is the intersection form, $2\delta$ is the divisor of non-reduced subschemes, and $(\delta,\delta)=2-2n$.  The embedding of $H^2(S,\Z)$ is achieved via the canonical isomorphism
\[H^2(S,\Z)\cong H^2(\Sym^nS,\Z)\]
and pullback along the contraction
$\sigma:S^{[n]}\into \Sym^nS$.  The inverse of
$(\cdot,\cdot)$ defines a $\Q$-valued form on
$H_2(X,\Z)$ which we will also denote
$(\cdot,\cdot )$; by Poincar\'{e} duality, we obtain
a decomposition dual to \eqref{decomp}.  For example, the
class $\delta^\vee\in H_2(X,\Z)$ Poincar\'{e} dual to the exceptional divisor $\delta$ has square $(\delta^\vee,\delta^\vee)=\frac{1}{2-2n}$.  The form induces an embedding $H^2(X,\Z)\subset H_2(X,\Z)$ under which the two forms match up, and since the determinant of $(\cdot,\cdot)$ on $H^2(X,\Z)$ is $2-2n$, we can write any $\ell\in H_2(X,\Z)$ as $\ell=\frac{\lambda}{2n-2}$ for some $\lambda\in H^2(X,\Z)$.  We will refer to the smallest multiple of $\ell$ that is in $H^2(X,\Z)$ as the Beauville-Bogomolov dual $\rho$ of $\ell$.

\subsection{Cones of effective curves}
Much of the geometry of a $K3$ surface $S$ is encoded in its nodal classes, the indecomposable effective curve classes $C$ for which $(C,C)=-2$.  Suppose $S$ has an ample divisor $H$; let $N_1(S,\mathbb{Z})\subset
H_2(S,\Z)$ be the group of curve classes modulo homological
equivalence, and $\NE_1(S)\subset N_1(S,\mathbb{R})=N_1(S,\mathbb{Z})\otimes \mathbb{R}$
the cone of effective curves.  Then it is well-know that \cite[Lemma 1.6]{torelli}
\begin{equation}\label{cone}\NE_1(S)=\langle C\in N_1(S,\Z)|
 H\cdot C>0\;\mathrm{and}\;C\cdot C\geq -2\rangle\end{equation}
By Kleiman's criterion there is a dual statement for the
ample cone; here by $\langle\cdots\rangle$ we mean ``the cone
generated by $\cdots$''.

Hassett and Tschinkel \cite{extremal} conjectured that
the cone of effective curves in a holomorphic symplectic variety
$X$ is similarly determined intersection theoretically by the
Beauville-Bogomolov form.  The original form of this conjecture was:
\setcounter{custom}{0}
\begin{conjecture_disproven}(\cite[Thesis 1.1]{extremal})\label{general}  Let $X$ be an irreducible holomorphic symplectic
variety with polarization $H$.  Then there is a positive rational constant $c_X$
dependent only on the deformation class of $X$ such that
\[\NE_1(X)=\langle C\in N_1(X,\Z)|H\cdot C>0\;\mathrm{and}\;
(C,C)\geq -c_X
\rangle\]
Further, if $X$ contains a smoothly embedded Lagrangian
$n$-plane $\P^n\subset X$, and $\ell\in \NE_1(X)$ is the
class of the line in $\P^n$, then the bound is realized:
\[(\ell,\ell)=-c_X\]
\end{conjecture_disproven}
\begin{remark} 
As stated the first part of this conjecture is false.  A counterexample was originally constructed by Markman \cite{privatemarkman} for $X$ deformation equivalent to a Hilbert scheme of 5 points on a $K3$ surface, and the example is treated in detail in \cite[Remark 9.4]{macri}.  It is still expected that the cone of effective curves is still cut out intersection theoretically in terms of the Beauville-Bogomolov form, though necessarily in a more complicated fashion.  For $X$ of dimension $<8$ the original form of the conjecture is still expected to be true.

%
\end{remark}
The sufficiency of the intersection theoretic criterion in Disproven Conjecture \ref{general} has been worked out in full detail for $X$
deformation equivalent to a Hilbert scheme of 2 points on a $K3$
surface \cite[Theorem 1]{moving}.  In this case, there are three types of indecomposable ``nodal'' classes that appear---those with Beauville-Bogomolov square $-\frac{1}{2},-2,$ and $-\frac{5}{2}$---and their extremal rays correspond
to the 2 types of extremal contractions:
\begin{itemize}
\item[(i)]  Divisorial extremal contractions.  In this case, the exceptional
divisor $E$ is contracted to a $K3$ surface $T$.  The generic fiber
over $T$ is either an $A_1$ or $A_2$ configuration of rational
curves \cite[Theorem 21]{moving}, and if $C$ is the class of the generic fiber of
the normalization, then either $(C,C)=-2$ or $-1/2$,
respectively.
\item[(ii)]  Small extremal contractions.  In this case, $f$ contracts a Lagrangian $\P^2$ to an isolated singularity, and the class of a line $\ell$ satisfies $(\ell,\ell)=-5/2$.
\end{itemize}
See \cite{extremal} for some speculations about the ``nodal" classes that appear in higher dimensions.

\subsection{Lagrangian $n$-planes}Generalizing slightly, let $X$ be an irreducible holomorphic symplectic \emph{manifold}---that is, a simply-connected K\"{a}hler manifold with $H^0(\Omega^2_X)\cong\C$ generated by a nowhere degenerate 2-form.
There are only two infinite families of deformation classes of irreducible holomorphic symplectic manifolds known:  Hilbert schemes of points on $K3$ surfaces and generalized Kummer varieties.  We will be concerned with the former; following Markman, we define $X$ to be of \K{n}-type if it is deformation equivalent to a Hilbert scheme of $n$ points on a $K3$ surface.  In this case, we expect: 
 \begin{conjecture_custom}(\cite[Conjecture 1.2]{extremal})\label{conj}  Let $X$ be of \K{n}-type, $\P^n\subset X$ a smoothly
 embedded Lagrangian $n$-plane, and $\ell\in
 H_2(X,\mathbb{Z})$ the class of the line in $\P^n$.  Then
 \[(\ell,\ell)=-\frac{n+3}{2}\]
 \end{conjecture_custom}
  Further, we expect this to still be the minimal Beauville-Bogomolov square of a reduced irreducible curve class.  The conjecture has been verified for $n=2$ in \cite{moving} and for $n=3$ in \cite{HHT}. 
 \begin{remark}There is a similar conjecture for the class of a
 line $\ell$ in a smoothly embedded Lagrangian $n$-plane
 $\P^n\subset X$ for $X$ deformation equivalent to a
 $2n$-dimensional generalized Kummer variety $K_nA$ of an abelian surface $A$.  In this case, we
 expect
 \[(\ell,\ell)=-\frac{n+1}{2}\]
 This conjecture has been verified for $n=2$ in \cite{kummer}.
 \end{remark}
 Our main result is a proof of Conjecture \ref{conj} in the $n=4$ case; furthermore, we completely classify the class of the Lagrangian $4$-plane:
 \begin{theorem}[see Theorem \ref{mainthm}]  Let $X$ be of \K{4}-type, $\P^4\subset X$ be a smoothly embedded Lagrangian 4-plane, $\ell\in H_2(X,\Z)$ the class of a line in $\P^4$, and $\rho=2\ell\in H^2(X,\Q)$.  Then $\rho$ is integral, and
 \[[\P^4]=\frac{1}{337920}\left(880\rho^4+1760\rho^2\chern{2}(X)-3520\theta^2+4928\theta
 \chern{2}(X)-1408\chern{2}(X)^2\right)\]
 Further, we must have $(\ell,\ell)=-\frac{7}{2}$.
 \end{theorem}
 Here $\theta$ is the image of the dual to the Beauville-Bogomolov form, thought of as an element of $\Sym^2H_2(X,\Q)^*\cong \Sym^2H^2(X,\Q)$, under the cup product map $\Sym^2H^2(X,\Q)\rightarrow H^4(X,\Q)$.  Likewise in the $n=3$ case the class of the Lagrangian 3-plane is completely determined by $\ell$, \emph{cf.} \cite[Theorem 1.1]{HHT}.  Our theorem provides evidence that Conjecture \ref{conj} is true in general, and conjecturally determines the minimal Beauville-Bogomolov square of indecomposable nodal classes on eightfolds deformation equivalent to Hilbert schemes of points on $K3$ surfaces.
 \subsection{Monodromy}
 We prove our result by using the representation theory of the monodromy group of $X$ to relate the intersection theory of $X$ to that of a Hilbert scheme of 4 points on a $K3$ surface, where the cohomology ring is actually computable.  In doing so we completely determine the ring of monodromy invariants on $X$.  
 
Recall that a monodromy operator is the parallel
translation operator on $H^*(X,\Z)$ associated to a smooth
family of deformations of $X$; the monodromy group $\Mon(X)$ is the subgroup of $\GL(H^*(X,\Z))$ generated by
all monodromy operators.  Let $\Mon^2(X)\subset \GL(H^2(X,\Z))$
be the quotient acting nontrivially on degree 2 cohomology, and
$\overline{\Mon}(X)\subset \GL(H^*(X,\C))$ (respectively $\overline{\Mon^2}(X)\subset \GL(H^2(X,\C))$) the Zariski closure of $\Mon(X)$ (respectively $\Mon^2(X)$).  By
the deformation invariance of the Beauville-Bogomolov form,
$\Mon^2(X)$ is actually contained in $\OO(H^2(X,\Z))$, the
orthogonal group of $H^2(X,\Z)$ with respect to
$(\cdot,\cdot)$.  \emph{A priori}, the full Lie group $G_X=\SO(H^2(X,\C))$ only acts on
$H^2(X,\C)$, but in fact for $X$ of \K{n}-type, the full cohomology ring $H^*(X,\C)$ carries a representation of $G_X=\SO(H^2(X,\C))$ compatible with cup product (\cite[Proposition 4.1]{HHT}).  The basic reason for this is two-fold, both results of Markman:
\begin{itemize}
\item[(a)] the quotient $\Mon(X)\into\Mon^2(X)$ has finite
kernel \cite[$\mathsection 4.3$]{autos};
\item[(b)] $G_X$ is a connected component of $\overline{\Mon^2}(X)$ \cite[$\mathsection 1.8$]{autos}.
\end{itemize}
The representation of $\Mon(X)$ on $H^*(X,\C)$ extends to one of $\overline\Mon(X)$.  By the above the connected component of the universal covers of $\overline{\Mon}(X),\overline{\Mon^2}(X)$ and $G_X$ are all identified, so the universal cover of $G_X$ acts on all of $H^*(X,\C)$; the representation descends to $G_X$ because of the vanishing of odd cohomology.

The action respects the Hodge structure, so we may consider the ring of Hodge classes:
\[I^*(X)=H^*(X,\Q)\cap H^*(X,\C)^{G_X}\]
Of course, $I^*(X)$ contains the Chern classes of the tangent bundle of $X$ and the Beauville-Bogomolov class $\theta\in H^4(X,\Q)$, but there can be many other Hodge classes.  Markman \cite{classes} constructs another series of Hodge classes $k_i\in I^{2i}(X)$, $i\geq2$, as characteristic classes of monodromy-invariant twisted sheaves.

Given $\lambda\in H^2(X,\Q)$, let $G_\lambda\subset G_X$ be the stabilizer of $\lambda$.  Define 
\[I_\lambda^*(X)=H^*(X,\Q)\cap H^*(X,\C)^{G_\lambda}\]
to be the ring of cohomology classes invariant under the monodromy group preserving $\lambda$.  For example, given a Lagrangian $n$-plane $\P^n\subset X$, the
deformations of $X$ that deform $\P^n$ are precisely those in
$H^{1,1}(X)\cap\rho^\perp$, where $\rho$ is the
Beauville-Bogomolov dual of the class of the line in $\P^n$, and
the orthogonal is taken with respect to the Beauville-Bogomolov
form \cite{hodge1,hodge2}.  Thus, the class $[\P^n]\in H^{2n}(X,\Z)$ must lie in the subring $I^*_\rho(X)$.  $G_X$ will act on these cohomology classes, and up to this action we expect there is a unique Lagrangian $n$-plane in general.  For $n=4$, this is a consequence of our result since $G_X$ acts transitively on rays in $H^2(X,\C)$:
\begin{corollary}\label{ring}For $X$ of \K{4}-type, there is a unique $G_X$ orbit of smooth Lagrangian $4$-plane classes $[\P^4]\in H^8(X,\C)$.
\end{corollary}
%
\subsection*{Method of Proof and Outline}
We prove our result by first completely determining $I^*_\lambda(X)$ for $X=S^{[4]}$ a Hilbert scheme of 4 points on a $K3$ surface $S$ and $\lambda=\delta$.  This is done in Section 1 using the Nakajima basis and the results of \cite{cupproduct} on cup product.  The ring $I^*_\lambda(X)$ in the general case of $X$ of \K{4}-type and $\lambda\in H^2(X,\Z)$ will be isomorphic since $G_X$ acts transitively on rays in $H^2(X,\Z)$.  In Section 2 we construct an explicit isomorphism by finding a  monodromy invariant basis for $I_\lambda^*(X)$, from which we are able to derive the intersection form on $I_\lambda^8(X)$.  In Section 3 we take $\lambda$ proportional to the Beauville-Bogomolov dual of the class of a line in a smooth Lagrangian $4$-plane $\P^4\subset X$ and produce a diophantine equation in the coefficients of the class $[\P^4]$ with respect to the basis from Section 2.  In Section 4, we show the only solution to the diophantine equation is the conjectural one.  For completeness we include an appendix summarizing our localization computations to calculate the Fujiki constants in Section 2.
 \subsection*{Acknowledgements}We are grateful to Y. Tschinkel for suggesting the problem, and for many insights.  We would also like to thank B. Hassett and M. Thaddeus for useful conversations, and M. Stoll for explaining to us how to compute integral points on elliptic curves in Magma.   Finally, we thank the referee for useful comments.  The first author was supported in part by NSF Fellowship DMS-1103982. This project was completed while the second author was a postdoc at the California Institute of Technology. Some computations were performed on William Stein's server \texttt{geom.math.washington.edu}, supported by NSF grant DMS-0821725.
 \section{Structure of the ring of monodromy invariants}

 \subsection{The Lehn-Sorger formalism}We briefly summarize the
 work of Lehn and Sorger in \cite{cupproduct} on the cohomology ring of a Hilbert
 scheme of points on a $K3$ surface.   Given a Frobenius algebra $A$, they construct a Frobenius algebra
 $A^{[n]}$ such that when $A=H^*(S,\Q)$ for $S$ a $K3$ surface,
 $A^{[n]}$ is canonically $H^*(S^{[n]},\Q)$. 

The algebra $A=H^*(S,\mathbb{Q})$ comes equipped with a form
 $T=-\int_S:A\into \Q$ and a multiplication $m:A\otimes A\into A$ (given by cup-product) such
 that the pairing $(x,y)=T(xy)$ is nondegenerate.  There is also
 a comultiplication $\Delta:A\into A\otimes A$ adjoint to $m$
 with respect to the form $T\otimes T$ on $A\otimes A$. In this
 case $\Delta$ is the push-forward along the diagonal. Writing
 $1\in H^0(S,\mathbb{Z})$ for the unit, $\pt\in
 H^4(S,\mathbb{Z})$ for the point class, $e_1,\ldots, e_{22}$ as
 a basis for $H^2(S,\mathbb{Z})$, and
 $e_1^\vee,\ldots,e_{22}^\vee$ for the dual basis with respect
 to the intersection form, a simple computation using
 adjointness shows that $\Delta(1)=-\sum_j e_j\otimes
 e_j^\vee-\pt\otimes1-1\otimes\pt$,
 $\Delta(e_j)=-e_j\otimes\pt-\pt\otimes e_j$,
 $\Delta(e_j^\vee)=-e_j^\vee\otimes\pt-\pt\otimes e_j^\vee$ and
 $\Delta(\pt)=-\pt\otimes\pt$. Thus $\euler=-24\pt$. 

 We also have an $n$-fold multiplication $m[n]:A^{\otimes n}\into A$ and its adjoint $\Delta[n]:A\into A^{\otimes n}$. Note that $m[1]=\Delta[1]=\id$, $m[2]=m$, and $\Delta[2]=\Delta$.
\begin{lemma}\label{delta2} Using the previous formulae one obtains:
\begin{align*}
\Delta[3](1)&=\sum_j\sum(e_j)_a\otimes
(e_j^\vee)_b\otimes\pt_c+\sum\pt_a\otimes\pt_b\otimes 1_c\\
\Delta[3](e_j)&=\sum \pt_a\otimes\pt_b\otimes (e_j)_c\\
\Delta[3](e_j^\vee)&=\sum\pt_a\otimes\pt_b\otimes (e_j^\vee)_c\\
\Delta[3](\pt)&=\pt\otimes\pt\otimes\pt
\end{align*}
\end{lemma}
By $\pt_a\otimes\pt_b\otimes 1_c\in A^{\otimes 3}$ we mean $\pt$ inserted in the $a$th and $b$th tensor factors, and $1$ inserted in the $c$th factor.  All unspecified sums in Lemma \ref{delta2} are over bijections $\{1,2,3\}\xrightarrow{\cong}\{a,b,c\}$.
\begin{proof}This follows from the relation $m[n]=m[2]\circ(m[n-1]\otimes\id)$ for $n\geq 2$ and the dual relation $\Delta[n]=(\Delta[n-1]\otimes\id)\circ\Delta[2]$.
\end{proof}

  Let $[n]=\{k\in
\mathbb{N}|k\leq n\}$.  Define the tensor
product of $A$ indexed by a finite set $I$ of cardinality $n$ as
\[A^I:=\left(\bigoplus_{\phi:[n]\xrightarrow{\cong}I}A_{\phi(1)}\otimes\cdots\otimes A_{\phi(n)}\right)/S_n\]
where $S_n$ acts by permuting the tensor factors in each summand
in the obvious way.  $A^I$ is a Frobenius algebra with
multiplication $m^I$ and form $T^I$.

Note that for (finite) sets $U,V$ and a
bijection $U\rightarrow V$ there is a canonical isomorphism $A^U\rightarrow
A^V$, so we can always choose a bijection of
$I$ with some $[k]$ to reduce
to the usual notion of finite self tensor products.
In general, for any surjection $\varphi:U\rightarrow V$, there is an obvious ring homomorphism
\[\phi^*:A^U\rightarrow A^V\]
using the ring structure to combine factors indexed by elements of $U$ in the same fiber of $\varphi$.  There is an adjoint map
\[\phi_*:A^V\rightarrow A^U\]
with the important relation
\[\phi_*(a\cdot\phi^*(b))=\phi_*(a)\cdot b\]
which follows directly from the adjointness.

For any subgroup $G\subset S_n$, we can
consider the left coset space $G\backslash[n]$, and form $A^{G\backslash[n]}$.  In particular, for
$\sigma\in S_n$ and $G=\langle\sigma\rangle$ the group generated
by $\sigma$, we denote $A^\sigma=A^{G\backslash[n]}$.  Let
\[A\{S_n\}=\bigoplus_{\sigma\in S_n} A^\sigma\cdot\sigma\]
A pure tensor element of $A^\sigma$ is
specified by attaching an element $\alpha_i\in A$ to each orbit
$i\in I=\langle \sigma\rangle\backslash[n]$.  For example, for a function $\nu:I\rightarrow
\mathbb{Z}_{\geq 0}$,
\[\euler^\nu=\otimes_{i\in I}\euler^{\nu(i)}\in A^\sigma\]

There is a natural product structure on $A\{S_n\}$.  For any
inclusion of subgroups $H\subset K$ of $S_n$ there is a
surjection $H\backslash[n]\rightarrow K\backslash[n]$ and
therefore maps
\[f^{H,K}:A^{H\backslash[n]}\rightarrow A^{K\backslash[n]}\]
\[f_{K,H}:A^{K\backslash[n]}\rightarrow A^{H\backslash[n]}\]
The product is then
\begin{equation}\label{multiplication}\begin{CD}
A^\sigma\otimes A^\tau@>>>A^{\sigma\tau}\\
a\otimes b@>>>f_{\langle\sigma,\tau\rangle,\langle\sigma\tau\rangle}\left(f^{\langle\sigma\rangle,\langle\sigma,\tau\rangle}(a)\cdot f^{\langle\tau\rangle,\langle\sigma,\tau\rangle}(b)\cdot \euler^{g(\sigma,\tau)}\right)
\end{CD}\end{equation}
where $\langle\sigma,\tau\rangle$ is the subgroup
of $S_n$ generated by $\sigma,\tau$, and the graph defect $g(\sigma,\tau):\langle\sigma,\tau\rangle\backslash[n]\rightarrow
\mathbb{Z}_{\geq 0} $ is
\[g(\sigma,\tau)(B)=\frac{1}{2}\left(|B|+2-|\langle\sigma\rangle\backslash
B|-|\langle\tau\rangle\backslash
B|-|\langle\sigma\tau\rangle\backslash B|\right)\]

$S_n$ acts naturally on $A\{S_n\}$.  For any $\tau\in A\{S_n\}$,
there is for any $\sigma\in S_n$ a bijection
$\tau:\langle\sigma\rangle\backslash[n]\into\langle\tau\sigma\tau^{-1}\rangle\backslash[n]$.
$\tau$ then acts on $A\{S_n\}$ via
$\tau^*:A^\sigma\cdot\sigma\into
A^{\tau\sigma\tau^{-1}}\cdot\tau\sigma\tau^{-1}$ on each factor.
Define
\[A^{[n]}=A\{S_n\}^{S_n}\]
Note that for any partition $\mu=(1^{\mu_1},2^{\mu_2},\cdots)$
of $n$, there is a piece
\begin{equation}A^{[n]}_\mu=\left(\bigoplus_{\sigma\in C_\mu}A^\sigma\cdot\sigma\right)^{S_n}\cong \bigotimes_i\Sym^{\mu_i}A\label{syms}\end{equation}
where $C_\mu\subset S_n$ is the conjugacy class of permutations
$\sigma$ of cycle type $\mu$.

If $A$ is a graded Frobenius algebra, then $A^{[n]}$ is
naturally graded.  $A^\sigma$ is graded as a tensor
product of graded vector spaces, and we take
\[A^\sigma\cdot\sigma\cong A^\sigma[-2|\sigma|]\]
where if the cycle type of $\sigma$ is $\mu$, $|\sigma|=\sum_i
(i-1)\mu_i$.  In particular, the $m$th graded piece of \eqref{syms} is
\begin{equation}(A^{[n]}_\mu)_{m}\cong
\bigoplus_{\substack{(w,\mu)\\|(w,\mu)|=m}}\bigotimes_i\Sym^{\mu_i}A_{w_i}\label{syms2}\end{equation}
where the sum is taken over  weighted permutations $(w,\mu)$---\emph{i.e.} a
partition $\mu$ and a weight $w_i$ associated to each part---with
\[m=|(w,\mu)|=\sum_i(i-1)\mu_i+w_i\]
We then have
\begin{theorem}(\cite[Theorem 1.1]{cupproduct})\label{isom}
For $S$ a $K3$ surface, there is a natural isomorphism of graded Frobenius algebras 
\[\left(H^*(S,\Q)[2]\right)^{[n]}\cong H^*(S^{[n]},\Q)[2n]\]
\end{theorem}
The grading shift on both sides is such that the 0th
graded piece is middle cohomology.
\begin{remark}It will be important in the next section to note
that under the isomorphism of Theorem \ref{isom},
\begin{equation}n!\pt_1\otimes\cdots\otimes\pt_n\cdot(\id)\mapsto \pt_{S^{[n]}}\label{pointclass}\end{equation}
\end{remark}
\subsection{Monodromy invariants}\label{monodromy}
Let $S$ be a $K3$ surface, and $G_S=\SO(H^2(S,\C))$ the special
orthogonal group of the intersection form $(\cdot,\cdot)$ on
$S$.  $H^*(S,\C)$ is naturally a representation of $G_S$, acting
via the standard representation on $H^2(S,\C)$ and the trivial
representations on $H^0(S,\C)$ and $H^4(S,\C)$.

Recall (see for example \cite{repn}) that positive weights of the algebra $\SO_\C(k)$ of rank $r$ ($k=2r$ or $2k+1$)
are $r$-tuples $\lambda=(\lambda_1,\ldots,\lambda_r)$ with the $\lambda_i$ either all integral or all half-integral, and either 
\begin{align*}
\lambda_1\geq\lambda_2\geq\cdots\geq \lambda_{r-1}\geq|\lambda_r|\geq 0,&\hspace{2em} k=2r\\
\lambda_1\geq\lambda_2\geq\cdots\geq \lambda_{r-1}\geq \lambda_r\geq0,&\hspace{2em} k=2r+1
\end{align*}
Let the representation of $\SO_\C$ of highest weight $\lambda$
be denoted $V(\lambda)$.  Thus, $\triv=V(0,\ldots)$ is the
trivial representation, and $V=V(1,0,\ldots)$ the standard.
$\Sym^kV$ is not irreducible, since the form yields an invariant
$\theta\in \Sym^2V$, but $V(k,0,\ldots)=\Sym^kV/\Sym^{k-2}V$.  In the sequel, we will only indicate the nonzero weights, \emph{e.g.} $V=V(1)$.

If a Frobenius algebra $A$ carries a representation of a group $G$, $A^{[n]}$ naturally carries a representation of $G$ that can easily be read off of \eqref{syms2}.  Thus,
\begin{proposition}As a representation of $G_S$, we have
\begin{align*}
H^2(S^{[4]},\C)&\cong \triv_S\oplus V_S(1)\\
H^4(S^{[4]},\C)&\cong \triv_S^4\oplus V_S(1)^2\oplus V_S(2)\\
H^6(S^{[4]},\C)&\cong\triv_S^5\oplus V_S(1)^5\oplus V_S(1,1)\oplus V_S(2)^2\oplus V_S(3)\\
H^8(S^{[4]},\C)&\cong\triv_S^8\oplus V_S(1)^6\oplus V_S(1,1)\oplus V_S(2)^4\oplus V_S(2,1)\oplus V_S(3)\oplus V_S(4)
\end{align*}
\end{proposition}
Poincar\'{e} duality is compatible with the $G_S$ action, so the above determines all cohomology groups.

Note that the invariant class in $H^2(S^{[n]},\C)$ is exactly $\delta$.  The decomposition \eqref{decomp} identifies the action of $G_S$ on $H^*(S^{[n]},\C)$ with that of $G_{\delta}\subset G_{S^{[n]}}$, the stabilizer of $\delta$.  In other words, deformations of $S^{[n]}$ orthogonal to the exceptional divisor $\delta$ remain Hilbert schemes of points of a $K3$ surface, and therefore come from a deformation of $S$.

Recall that $\SO_\C(k)$ has universal branching rules.  For
$\SO_\C(k-1)\subset\SO_\C(k)$ the stabilizer of a nonisotropic
vector $v\in V$, $(v,v)\neq 0$, we have
\[\Res^{\SO_\C(k)}_{\SO_\C(k-1)}V(\lambda)=\bigoplus_{\lambda'}V(\lambda')\]
where the sum is taken over all weights $\lambda'$ with
\[\lambda_1\geq\lambda_1'\geq \lambda_2\geq
\lambda_2'\geq\cdots\geq\lambda_r\geq|\lambda_r'|\geq 0\]

For $X$ of \K{n}-type, we can therefore deduce the structure of
$H^*(X,\C)$ as a $G_X$ representation from the structure of
$H^*(S^{[n]},\C)$ as a $G_S$ representation:
\begin{corollary} For $X$ of \K{4}-type, 
\begin{align*}
H^2(X,\C)&\cong V_X(1)\\
H^4(X,\C)&\cong \triv_X^2\oplus V_X(1)\oplus V_X(2)\\
H^6(X,\C)&\cong\triv_X\oplus V_X(1)^2\oplus V_X(1,1)\oplus V_X(2)\oplus V_X(3)\\
H^8(X,\C)&\cong\triv_X^3\oplus V_X(1)^2\oplus V_X(2)^2\oplus V_X(2,1)\oplus V_X(4)
\end{align*}
\end{corollary}
Again, Poincar\'{e} duality determines the representations of the other cohomology groups.
 \subsection{A basis for $I^*_\delta(S^{[4]})$} 
 For a partition $\mu=(1^{\mu_1},2^{\mu_2},\ldots)$ of $n$, the number of parts of $\mu$ is $\ell(\mu)=\sum \mu_i$.  By a labelled
partition $\mathbf{\mu}$ we will mean a partition $\mu$ and an
ordered list of $\ell(\mu)$ cohomology classes $\alpha\in
H^*(S,\Q)$.  For example, $(\{1\}_2,\{1,1\}_1)$ is a labelled
partition of 4, subordinate to the partition $\mu=(1^2,2)$, and
attaching the unit class to each part of $\mu$.  Such a labelled
partition $\mathbf{\mu}$ determines an element of
the Lehn-Sorger algebra of $H^*(S,\Q)[2]$ by summing over all
group elements $\sigma\in S_n$ with cycle type $\mu$, for
example
\begin{align*}
I(\{1\}_2,\{1,1\}_1)&=\sum_{(12)}1_{12}\otimes1_3\otimes1_4(12)\\
&=1_{12}\otimes 1_3\otimes 1_4(12)+1_{13}\otimes 1_2\otimes
1_4(13)+1_{14}\otimes 1_2\otimes 1_3(14)\\
&+1_1\otimes1_{23}\otimes 1_4(23)+1_{1}\otimes 1_{24}\otimes 1_3(24)+1_{1}\otimes 1_2\otimes 1_{34}(34)
\end{align*}
We can generate homogeneous classes of $H^*(S^{[n]},\Q)$ invariant under $G_S$ from partitions of $n$ labelled by cohomology classes $\{1,e,e^\vee,\pt\}$, where every time we have a label $e$, there must be a paired $e^\vee$ label, corresponding to inserting $e_j$ and $e_j^\vee$ in the corresponding tensor factors and summing over $j$.  For example, $I^2_\delta(S^{[4]})$ is spanned by $\delta=I(\{1\}_2, \{1, 1\}_1)$.  Generating sets for $I_\delta^{2k}(S^{[4]})$ for $k=2,3,4$ are given by:
\begin{center}
\begin{tabular}{rl|rl|rl}
\multicolumn{2}{c|}{$I^4_\delta(S^{[4]})$}		&	\multicolumn{2}{c}{$I^6_\delta(S^{[4]})$}		&	\multicolumn{2}{|c}{$I_\delta^8(S^{[4]})$}\\
\hline
$W=$&$I(\{1\}_3, \{1\}_1)$	&	$P=$&$I(\{1\}_4)$			&	$A=$&$I(\{e\}_3,\{e^\vee\}_1)$\\
$X=$&$I(\{1,1\}_2)$			& 	$Q=$&$I(\{\pt\}_2, \{1,1\}_1)$	&	$B=$&$I(\{1\}_3,\{\pt\}_1)$\\
$Y=$&$I(\{1, 1,1,\pt\}_1)$		&	$R=$&$I(\{1\}_2,\{1,\pt\}_1)$	&	$C=$&$I(\{\pt\}_3,\{1\}_1)$\\
$Z=$&$I(\{1,1,e,e^\vee\}_1)$	&	$S=$&$I(\{e^\vee\}_2,\{e,1\}_1)$&	$D=$&$I(\{1,\pt\}_2)$\\
&						&	$T=$&$I(\{1\}_2,\{e,e^\vee\}_1)$&	$E=$&$I(\{e,e^\vee\}_2)$\\
&&&&$F=$&$I(\{1,1,\pt,\pt\}_1)$\\
&&&&$G=$&$I(\{1,e,e^\vee,\pt\}_1)$\\
&&&&$H=$&$I(\{e,e,e^\vee,e^\vee\}_1)$\\
\end{tabular}
\end{center}
These classes are all clearly independent, and therefore by the computation of the dimensions of $I^*_\delta(S^{[4]})$ in the previous section they are bases.
\subsection{Cup product on $I^*_\delta(S^{[4]})$}
Using \eqref{multiplication} we compute the multiplicative
structure of $I^*_\delta(S^{[4]})$ in the above basis. These
computations are straightforward; for example,
\begin{align*}
\delta^2&=\left(\sum_{(12)}1_{12}\otimes 1_3\otimes
1_4(12)\right)^2\\
&=\sum_{(12)}\Bigl(\Delta(1)_{1,2}\otimes 1_3\otimes
1_4(\id)+1_{1,2,3}\otimes 1_4(132)\\
&\hspace{.5in}+1_{1,2,4}\otimes 1_3(142)+1_{1,2,3}\otimes
1_4(123)+1_{1,2,4}\otimes 1_3(124)+1_{12}\otimes
1_{34}(12)(34)\Bigl)\\
&=-3\sum_{1}\pt_1\otimes 1_2\otimes 1_3\otimes
1_4(\id)-\sum_{(12)}\sum_j(e_j)_1\otimes (e_j^\vee)_2\otimes
1_3\otimes 1_4(\id)\\
&\hspace{.5in}+3\sum_{(123)}1_{123}\otimes
1_4(123)+2\sum_{(12)(34)}1_{12}\otimes 1_{34}(12)(34)\\
&=-3Y-Z+3W+2X
\end{align*}The multiplication table for degree 4 elements is:
\begin{center}
\begin{tabular}{c||c|c|c|c}
&$W$&$X$&$Y$&$Z$\\
\hline\hline
$W$&$\substack{-3A-3B-27C-8D\\-8E+4F+2G}$&$-3A-3B-3C$&$B+3C$&$3A+66C$\\
\hline
$X$&&$-2D-2E+2F+G+H$&$2D$&$22D+4E$\\
\hline
$Y$&&&$2F$&$G$\\
\hline
$Z$&&&&$22F+2G+2H$\\
\end{tabular}
\end{center}
In particular, note that:
\begin{align}
\delta^4&=(\delta^2)^2=-81A- 81B- 729C- 192D-96E+ 84F + 30G + 6H\label{d4}
\end{align}
The multiplication table for $A,B,C,D,E,F,G,H$ is much simpler, 
\begin{center}
\begin{tabular}{c||c|c|c|c|c|c|c|c}
&$A$&$B$&$C$&$D$&$E$&$F$&$G$&$H$\\
\hline\hline
$A$&$\frac{176}{24}$&0&0&0&0&0&0&0\\
\hline
$B$&&0&$\frac{8}{24}$&0&0&0&0&0\\
\hline
$C$&&&0&0&0&0&0&0\\
\hline
$D$&&&&$\frac{6}{24}$&0&0&0&0\\
\hline
$E$&&&&&$\frac{66}{24}$&0&0&0\\
\hline
$F$&&&&&&$\frac{6}{24}$&0&0\\
\hline
$G$&&&&&&&$\frac{264}{24}$&0\\
\hline
$H$&&&&&&&&$\frac{1584}{24}$\\
\end{tabular}
\end{center}
where we have identified top cohomology $H^{16}(S^{[4]},\Q)\cong\Q$ as usual via the point class $\pt_{S^{[4]}}=24\pt_1\otimes\pt_2\otimes\pt_3\otimes\pt_4(\id)$ from \eqref{pointclass}. As a consistency check, from Corollary \ref{fujikis} we have $\delta^8=105(\delta,\delta)^4=136080$ and indeed, from \eqref{d4}, $\delta^8=(-81A - 81B - 729C - 192D - 96E + 84F + 30G + 6H)^2=136080$.
Note that the remaining classes and products (of cohomological degree divisible by 4, which is all we need) are determined by
Poincar\'{e} duality.
\subsection{The Beauville-Bogomolov form}From \eqref{decomp}, we
can explicitly write down $\theta$ in the $W,X,Y,Z$ basis:
\begin{align}
\theta&=\sum_j\left(\sum_{1}(e_j)_1\otimes1_2\otimes1_3\otimes1_4(\id)\right)\cdot\left(
\sum_{1}(e_j^\vee)_1\otimes1_2\otimes1_3\otimes1_4(\id)\right)-\frac{1}{6}\delta^2\notag\\
&=-\frac{1}{2}W-\frac{1}{3}X+\frac{45}{2}Y+\frac{13}{6}Z\label{bogoform}
\end{align}
By direct compoutation, using the results of the previous
section,
\begin{lemma}\label{theta}
\begin{align*}
\theta^4&=450225\\
\delta^2\theta^3&=-117450=19575(-6)\\
\delta^4\theta^2&=84564=2349\cdot(-6)^2\\
\delta^6\theta&=-93960=435\cdot(-6)^3\\
\delta^8&=136080=105\cdot(-6)^4
\end{align*}
\end{lemma}
\section{Hodge classes on $X$}
Let $X$ be of \K{4}-type and $\lambda\in H^2(X,\Q)$.  The rings $I^*(X)$ and $I^*_\lambda(X)$ are isomorphic to the rings $I^*(S^{[4]})$ and $I^*_\delta(S^{[4]})$ since the action of $G_X$ is transitive on rays, but to construct an explicit isomorphism, we must find a geometric basis.  To do this, we need to understand the products of Hodge classes.
\subsection{Computation of the Fujiki constants for $S^{[4]}$}\label{fujikisection}
 Let $X$ be smooth variety of dimension $n$, and $\mu$
a partition of a nonegative integer $|\mu|$ (we allow the empty
partition of $0$).  To each $\mu$ we can associate a Chern monomial
$\chern{\mu}(X)=\prod_{i=1}^k\chern{k}^{\mu_k}(X)$. 
Given a formal power series $\phi(x)\in\Q[[x]]$, define the associated genus
\[\phi(X)=\prod_i\phi(x_i)\in H^*(X,\Q)\]
where the $x_i$ are the Chern roots of the tangent bundle $TX$.  Taking the universal formal power series
\[\Phi(x)=1+a_1x+a_2x^2+\cdots\in \Q[a_1,a_2,\ldots][[x]]\]
we define the universal genus $\Phi(X)$ of any smooth variety as
an element of $H^*(X,\Q)[a_1,a_2\ldots]$.  $\Phi(X)$
is a universal formal power series in the Chern classes
$\chern{1},\chern{2},\ldots$ with coefficients polynomials in
$a_1,a_2,\ldots$.  In particular, taking $a_1=1$ and $a_i=0$ for $i>1$, we get the total Chern class.  We will only need the universal genus for vanishing odd Chern classes; the reader may find the expansion of $\Phi$ in this case up to degree 16 in the appendix.


Let $S$ be a smooth surface, $\phi(x)\in\Q[[x]]$ a formal power series in $x$.  Recall that $\O^{[n]}$ is the push-forward of the structure sheaf of the universal subscheme $Z\subset S\times S^{[n]}$ to $S^{[n]}$, and that $\det\O^{[n]}=-\delta$.  A result of \cite[Theorem 4.2]{lehn} implies that there are universal formal power
series $A(z),B(z)$ in $z$ such that
\[\sum_{n\geq
  0}z^n\int_{S^{[n]}}\exp(\det\O^{[n]})\phi(S^{[n]})=A(z)^{\chern{1}(S)^2}B(z)^{\chern{2}(S)}\]
  for any smooth surface $S$.  Let
\[\mathbf{F}_S(z)=\sum_{n\geq0}z^n\int_{S^{[n]}}\exp(\det\O^{[n]})\Phi(S^{[n]})\in \Q[a_1,a_2,\ldots][[z]]\]
and let $\mathbf{A}(z),\mathbf{B}(z)\in\Q[a_1,a_2,\ldots][[z]]$ be the universal power series associated to $\Phi$.  $\mathbf{F}_{\P^2}(z)=\mathbf{A}(z)^9\mathbf{B}(z)^3$ and $\mathbf{F}_{\P^1\times\P^1}(z)=\mathbf{A}(z)^8\mathbf{B}(z)^4$ can be easily computed by routine equivariant localization and therefore one can compute $\mathbf{A}(z),\mathbf{B}(z)$; see the appendix for a brief summary of the computation.
Since $\P^1\times\P^1,\P^2$ generate the cobordism ring, this determines $\mathbf{F}_S(z)$ for a $K3$ surface $S$, and in particular we can compute all products
\begin{equation}\int_{S^{[n]}}\delta^{2k}\chern{\mu}(S^{[n]})
\label{products}\end{equation}
By the following result of Fujiki, \eqref{products} determines all products of the form
\[\int_Xf^{2k}\chern{\mu}(X)\]
for arbitrary $f\in H^2(X,\Q)$:
\begin{theorem}\label{fujiki}\cite{fujiki}  For $X$ an irreducible holomorphic symplectic variety of dimension $n$ and $\mu$ an even partition of an integer $|\mu|$, there are 
rational constants $\gamma_X(\mu)$ such that, for any class $f\in
H^2(X,\mathbb{Z})$,
\[\int_Xf^{2k}\chern{\mu}(X)=\gamma_X(\mu)\cdot(f,f)^{k},\indent |mbox{for}\indent 2k=2n-|\mu|\]

Moreover, the constant $\gamma_X(\mu)$ is a deformation invariant.
\end{theorem}
Of course, if $|\mu|>\dim X$, we have $\gamma_X(\mu)=0$. Also, because $X$ is holomorphic symplectic, all odd Chern
classes $\chern{i}(X)$ vanish, so we require $\mu$ to be an even partition.  We collect here the Fujiki
constants $\gamma(\mu)$ for $n=4$ for reference:
\begin{corollary}\label{fujikis}  For $X$ of \K{4}-type, we have
\[\begin{array}{rlrlrlll}
\gamma_X(2^4)=&1992240&\gamma_X(2^3)=&59640&\gamma_X(2^2)=&4932&\gamma_X(2^1)=630&\gamma_X(\varnothing)=105\\
\gamma_X(2^24^1)=&813240&\gamma_X(2^14^1)=&24360&\gamma_X(4^1)=&2016&&\\

\gamma_X(2^16^1)=&182340&\gamma(6^1)=&5460&&&\\
\gamma_X(8^1)=&25650&&&&&&\\
\gamma_X(4^2)=&332730&&&&&&\\
\end{array}\]
\end{corollary}
\begin{proof}This follows from the deformation invariance and
the degree 4 part of
\[\mathbf{F}_S(z)=\mathbf{B}(z)^{24}\]
for $S$ a $K3$ surface.  Note that $(\delta,\delta)=-6$.
\end{proof}
\begin{remark}  The first column of numbers are the Chern numbers of $X$, and were computed in \cite{lehn}; $\gamma_X(\varnothing)$ is the ordinary Fujiki constant.  The authors are unaware of a computation of the middle three columns in the literature.
\end{remark}
\subsection{Generalized Fujiki constants}Let $X$ be of \K{n}-type.  In general, for $\eta$ a Hodge class, an integral of the form $\int_X f^{2k}\eta$ must be compatible with the $G_X$ action, and therefore will be a rational multiple of $(f,f)^k$.  For $\eta$ a product of a power of $\theta$ and a Chern monomial, these ratios are determined by the Fujiki constants of the previous section.

Define an augmented partition $(\ell,\mu)$ to be a partition
$\mu$ of a nonnegative integer $|\mu|$ and a nonegative integer $\ell$.  Set
\[|(\ell,\mu)|=2\ell+|\mu|\]
 \begin{proposition}\label{genfujiki}For $X$ of \K{n}-type, $n>1$, and $(\ell,\mu)$ an augmented even partition, there is a
 rational constant $\gamma_X(\ell,\mu)$ such that for any $f\in H^2(X,\mathbb{Z})$,
 \[\int_Xf^{2k}\theta^\ell\chern{\mu}(X)=\gamma_X(\ell,\mu)\cdot(f,f)^{k}, \indent \mbox{for}\indent  2k=2n-2\ell-|\mu|\]
Furthermore, there are rational constants $\alpha(k,\ell)$
independent of $X$ such that
\[\gamma_X(\ell,\mu)=\alpha(k,\ell)\gamma_X(\mu), \indent\mbox{for}\indent 2k=2n-2\ell -|\mu|\]
\end{proposition}
Again, $\gamma_X(k,\ell,\mu)=0$ if $|(\ell,\mu)|> \dim X$.
\begin{proof}As mentioned above, the interesting part is the
existence of the $\alpha$.  Let $x_i$ be an orthonormal basis of $H^2(X,\C)$ with respect to the Beauville-Bogomolov form.  Note that $\theta = \sum_i x_i^2$.  It suffices to consider the case $f=\sum_i x_i$, which has $(f,f)=23$.  Let 
\[p^{k}(a)=\left(\sum_i a_ix_i\right)^k\]
for $a\in\Q^{23}$.  The $p^k(a)$ span the space of degree
$k$ polynomials in $x_i$, so their
symmetrizations
\[\overline{p}^k(a)=\frac{1}{23!}\sum_{\sigma\in
  S_{23}}\left(\sum_i a_ix_{\sigma(i)}\right)^k\]
span the space of degree $k$ symmetric functions in $x_i$.  We
can therefore write
\[f^{2k}\theta^\ell=\sum_{a(k,\ell)}\lambda_{a(k,\ell)}\overline{p}^{2k+2\ell}(a(k,\ell))\]
where the sum is over finitely many $a(k,\ell)$.  This expression has no dependence on the dimension of $X$. We have
\begin{align*}
\int_Xf^{2k}\theta^\ell\chern{\mu}(X)&=\frac{1}{23!}\sum_{a(k,\ell)}\lambda_{a(k,\ell)}\sum_{\sigma\in
  S_{23}}\int_X \left(\sum_i a(k,\ell)_ix_{\sigma(i)}\right)^{2k+2\ell}\chern{\mu}(X)\\
&=\frac{1}{23!}\sum_{a(k,\ell)}\lambda_{a(k,\ell)}\sum_{\sigma\in
  S_{23}}\left(\sum_i
a(k,\ell)_ix_{\sigma(i)},
\sum_i
a(k,\ell)_ix_{\sigma(i)}
\right)^{k+\ell}\gamma_X(\mu)\\
&=\left(\sum_{a(k,\ell)}\lambda_{a(k,\ell)}\left(\sum_i
a(k,\ell)_i^2\right)^{k+\ell}\right)\gamma_X(\mu)\\
&=\alpha(k,\ell)\gamma_X(\mu)(f,f)^{k+\ell}
\end{align*}
where
\[\alpha(k,\ell)=\frac{1}{23^k}\sum_{a(k,\ell)}\lambda_{a(k,\ell)}\left(\sum_i
a(k,\ell)_i^2\right)^{k+\ell}
\]
\end{proof}
Explicitly,
\[\int_X \theta\chern{\mu}(X)=\sum_i\int_X x_i^2\chern{\mu}(X)=\sum_i(x_i,x_i)\gamma_X(\mu)=23\cdot\gamma_X(\mu)\]
so $\alpha(0,1)=23$.  Less trivially,
\begin{align*}
\int_X\theta^2\chern{\mu}(X)&=\int_X\left(\sum_i
x_i^2\right)^2\chern{\mu}(X)\\
&=\int_X\left(\frac{1}{6}\sum_{i<j}(x_i+x_j)^4+\frac{1}{6}\sum_{i<j}(x_i-x_i)^4-\frac{19}{3}\sum_ix_i^4\right)\chern{\mu}(X)=\frac{575}{3}\cdot\gamma_X(\mu)
\end{align*}
The relevant values of the $\alpha$ constants can be computed from Lemma \ref{theta} and by reducing to the \K{3}-type case:
 \begin{lemma}\label{constants}We have
 \[\begin{array}{llll}
 \alpha(0,1)=23&\alpha(1,1)=\frac{25}{3}&\alpha(2,1)=\frac{27}{5}&\alpha(3,1)=\frac{29}{7}\\
 \alpha(0,2)=\frac{575}{3}&\alpha(1,2)=45&\alpha(2,2)=\frac{783}{35}&\\
 \alpha(0,3)=1035&\alpha(1,3)=\frac{1305}{7}&&\\
\alpha(0,4)=\frac{30015}{7}&&
\end{array}\]
 \end{lemma}
Of course, $\alpha(k,0)=1$ for any $k$.
 \begin{proof} $\alpha(3,1),
 \alpha(2,2),\alpha(1,3),\alpha(0,4)$ are all determined by
 Lemma \ref{theta}, using $\gamma_{S^{[4]}}(\varnothing)=105$.  Because $\alpha(k,\ell)$ is independent of the dimension of $X$, we can determine the remaining $\alpha$ constants from
 the computations of \cite{HHT} in the \K{3}-type cases, where
 \begin{align*}
 (\theta_{S^{[3]}})^3=15525&=1035\cdot\gamma_{S^{[3]}}(\varnothing)\\
 (\delta_{S^{[3]}})^2(\theta_{S^{[3]}})^2=-2700&=45\cdot(\delta_{S^{[3]}},\delta_{S^{[3]}})\cdot\gamma_{S^{[3]}}(\varnothing)\\
 (\delta_{S^{[3]}})^4(\theta_{S^{[3]}})=1296&=\frac{27}{5}\cdot(\delta_{S^{[3]}},\delta_{S^{[3]}})^2\cdot\gamma_{S^{[3]}}(\varnothing)\\
 (\theta_{S^{[3]}})^2\chern{2}(S^{[3]})=20700&=\frac{575}{3}\cdot\gamma_{S^{[3]}}(2^1)\\
 (\delta_{S^{[3]}})^2(\theta_{S^{[3]}})\chern{2}(S^{[3]})=-3600&=\frac{25}{3}\cdot(\delta_{S^{[3]}},\delta_{S^{[3]}})\cdot\gamma_{S^{[3]}}(2^1)
\end{align*}
since $\gamma_{S^{[3]}}(\varnothing)=15, \gamma_{S^{[3]}}(2^1)=108$ and $\chern{2}(S^{[3]})=\frac{4}{3}\theta_{S^{[3]}}$.

\end{proof}
\subsection{A geometric basis}
$I^8(X)$ is 3-dimensional, so we expect there to be a relation among $\theta^2,\theta\chern{2}(X),\chern{2}(X)^2,\chern{4}(X)$:
\begin{lemma}  For $X$ of \K{4}-type,
\begin{equation}\theta^2=\frac{7}{5}\theta c_2-\frac{31}{60}c_2^2+\frac{1}{15}c_4\label{reln}\end{equation}
\end{lemma}
\begin{proof}Using the results of the previous section, we know the
 intersection form restricted to $I^8(X)$ in terms of the basis $\theta^2,\theta\chern{2}(X),\chern{2}(X)^2,\chern{4}(X)$:
\begin{equation}\begin{pmatrix}
450225&1035\cdot630&\frac{575}{3}\cdot4932&\frac{575}{3}\cdot2016\\
1035\cdot 630&\frac{575}{3}\cdot4932&23\cdot 59640&23\cdot 24360\\
\frac{575}{3}\cdot4932&23\cdot 59640&1992240&813240\\
\frac{575}{3}\cdot2016&23\cdot 24360&813240&332730
\end{pmatrix}\label{kernel}\end{equation}
As expected, the matrix is rank 3.  By Poincar\'{e} duality, a generator of the kernel gives the relation.
\end{proof}

\begin{corollary}\label{chern}
$\chern{2}(S^{[4]})=3Z+33Y-W$
\end{corollary}
\begin{proof}
Suppose $\chern{2}(S^{[4]})=wW+xX+yY+zZ$ for $w,x,y,z\in\Q$. Taking the product with
$\theta^3,\delta^2\theta^2,\delta^4\theta,\delta^4$ yields the
equation
\[\begin{pmatrix}
-6075&-2700&\frac{30375}{2}&\frac{96525}{2}\\
15066&6696&-3213&-16335\\
-19116&-8496&1854&14058\\
29160&12960&-1620&-17820
\end{pmatrix}
\begin{pmatrix}
w\\x\\y\\z
\end{pmatrix}
=\begin{pmatrix}
652050\\-170100\\122472\\-136080
\end{pmatrix}
\]
The matrix has rank 2.  Computing generators of the kernel, we can write
\[\chern{2}(S^{[4]})=\left(-\frac{4}{9}u-\frac{4}{27}v\right)W+\left(u-\frac{21}{4}\right)X+(v+42)Y-\frac{v}{3}Z\]
Similarly, computing $\chern{2}(S^{[4]})^2$ and intersecting with
$\theta^2,\delta^2\theta,\delta^4$ yields 3 equations: $945300=\theta^2\chern{2}(S^{[4]})^2$, $-246600=\theta\delta^2\chern{2}(S^{[4]})$ and $177552=\delta^4\chern{2}(S^{[4]})$ which have exactly two common solutions: $(u,v)=(\frac{21}{4},-9),(\frac{497}{116},-\frac{285}{29})$.
Finally, only one of these solutions, $(u,v)=(\frac{21}{4},-9)$, satisfies the additional equation $
1992240=\chern{2}(S^{[4]})^4$, and this gives the desired equation.
\end{proof}
Recall that $I^{4}(X)$ is 2-dimensional, whereas $I^{4}_\lambda(X)$ is 4-dimensional.  We already have $\lambda^2\in I^{4}_\lambda(X)$.  We need one more geometrically defined class in $I^4_\lambda(X)$ independent from $\lambda^2$ and $I^4(X)$ to get a basis for $I^4_\lambda(X)$: 
\begin{definition}\label{lambdaness}
Given a class $\lambda\in  H^2(X,\Q)$ (with $(\lambda,\lambda)\neq 0$ so no power of $\lambda$ is zero), define $\alpha\in I^4_\lambda(X)$ by Poincar\'{e} duality to be the unique class (up to a multiple) that intersects trivially with $\lambda^6$ and $I^{12}(X)$.
\end{definition}
\begin{lemma}For $X=S^{[4]}$ and $\lambda=\delta$, we may take $\alpha=X-3Y+Z$ which intersects trivially with
\[\delta^4\theta,\delta^4\chern{2}(S^{[4]}),\delta^2\theta^2,\delta^2\theta\chern{2}(S^{[4]}),\delta^2\chern{2}(S^{[4]})^2,\theta^3,\theta^2\chern{2}(S^{[4]}),\theta\chern{2}(S^{[4]})^2,\chern{2}(S^{[4]})^3\]
Further, $\alpha^2\theta^2=9450$, $\alpha^2\theta\chern{2}(S^{[4]})=14148$ and $\alpha^2\chern{2}(S^{[4]})^2=21168$.
\end{lemma}
\begin{proof}  By intersecting with $\theta$ and $\chern{2}(S^{[4]})$
using Corollary \ref{fujikis} and Lemma \ref{constants}, we see that $\theta^3$ and
$\theta^2\chern{2}(S^{[4]})$ are independent in $I^{12}(S^{[4]})$, so it is enough to show that $\alpha$ intersects these two classes to conclude it intersects trivially with each of the four degree 12 Hodge classes at the end of the list.  This, along with all the other claimed products, follow from Corollary \ref{chern}, equation \eqref{bogoform}, and our knowledge of the product structure.  Indeed,  
\begin{align*}
\alpha^2&=-3G + 30D + 42F + 3H + 6E\\
\alpha\delta^2&=-18B+162C\\
\alpha\theta&=88D + 8E - 27C + 3B - 88F + 20G + 4H\\
\alpha\chern{2}(S^{[4]})&=-54C + 132D + 6B + 12E + 30G +
6H - 132F\\
\theta^2&=-8E + \frac{19}{2}H + \frac{215}{2}G - 64D - \frac{33}{4}A - \frac{97}{4}B +
1117F - \frac{873}{4}C\\
\delta^4&=-81A- 81B- 729C- 192D-96E+ 84F + 30G + 6H\\
\theta\chern{2}(S^{[4]})&=-\frac{27}{2}A -
\frac{747}{2}C - \frac{83}{2}B - 8E +
1630F + 153G + 13H - 48D\\
\chern{2}(S^{[4]})^2&=18H - 8E - 69B - 8D + 218G - 21A - 621C + 2380F
\end{align*}
and the pairwise products are easily computed.  

\end{proof}
Because the cup-product structure on $H^*(S^{[4]},\Z)$ is preserved under deformation, and the monodromy  group acts transitively on rays in $H^2(S^{[4]},\Q)$, we immediately conclude the same for arbitrary $\lambda$
:
\begin{corollary}\label{alpha}For $\alpha$ chosen as in Definition \ref{lambdaness} with respect to
$\lambda\in H^2(X,\Z)$, $\alpha$ intersects trivially with
\[\lambda^4\theta,\lambda^4\chern{2}(X),\lambda^2\theta^2,\lambda^2\theta\chern{2}(X),\lambda^2\chern{2}(X)^2,\theta^3,\theta^2\chern{2}(X),\theta\chern{2}(X)^2,\chern{2}(X)^3\]
Further, up to a rational square, $\alpha^2\theta^2=9450$, $\alpha^2\theta\chern{2}(X)=14148$ and $\alpha^2\chern{2}(X)^2=21168$.
\end{corollary}
\subsection{Middle cohomology}
Putting Lemma \ref{constants} and Corollaries \ref{fujikis}, and \ref{alpha} together, we now know the complete intersection form on middle
cohomology $I^8_\lambda(X)$ with respect to the basis:
\begin{equation}\label{basis}\lambda^4,\lambda^2\theta,\lambda^2\chern{2}(X),\theta^2,\theta
\chern{2}(X),\chern{2}(X)^2,\alpha\theta,\alpha\chern{2}(X)\end{equation}
Denoting it by $M(\lambda)$, it is:
\[\begin{pmatrix}
105(\lambda,\lambda)^4&435(\lambda,\lambda)^3&630(\lambda,\lambda)^3&2349(\lambda,\lambda)^2&3402(\lambda,\lambda)^2&4932(\lambda,\lambda)^2&&\\
435(\lambda,\lambda)^3&2349(\lambda,\lambda)^2&3402(\lambda,\lambda)^2&19575(\lambda,\lambda)&28350(\lambda,\lambda)&44110(\lambda,\lambda)&&\\
630(\lambda,\lambda)^3&3402(\lambda,\lambda)^2&4932(\lambda,\lambda)^2&28350(\lambda,\lambda)&44110(\lambda,\lambda)&59640(\lambda,\lambda)&&\\
2349(\lambda,\lambda)^2&19575(\lambda,\lambda)&28350(\lambda,\lambda)&450225&652050&945300&&\\
3402(\lambda,\lambda)^2&28350(\lambda,\lambda)&41100(\lambda,\lambda)&652050&945300&1371720&&\\
4932(\lambda,\lambda)^2&41100(\lambda,\lambda)&59640(\lambda,\lambda)&945300&1371720&1992240&&\\
&&&&&&9450&14148\\
&&&&&&14148&21168\\
\end{pmatrix}\]
Note that this matrix is nonsingular if $(\lambda,\lambda)\neq0$, and therefore \eqref{basis} is in fact a basis.
 \section{Lagrangian $n$-planes in $X$}
 Let $X$ be a $2n$ dimensional holomorphic symplectic variety,
 and suppose that $\P^n\subset X$ is a smoothly embedded
 Lagrangian $n$-plane.  By a simple calculation,
 \begin{lemma}\cite{HHT}\label{restrict}Denote by $h$ the hyperplane class on $\P^n$.
 Then in the above setup,
 \[\chern{2j}(T_X|_{\P^n})=(-1)^{j}h^{2j}\binom{n+1}{j}\]
 \end{lemma}
 \begin{proof}
 We have
 \[0\into T_{\P^n}\into T_X|_{\P^n}\into N_{\P^n/X}\into0\]
 and since $\P^n$ is Lagrangian, $N_{\P^n/X}\cong T^*_{\P^n}$, so
 \[\chern{}(T_X|_{\P^n})=(1+h)^{n+1}(1-h)^{n+1}=(1-h^2)^{n+1}\]
 \end{proof}
 Let $\theta$ be the Beauville-Bogomolov class.  Then for $n=4$,
 \begin{lemma}\label{thetares}$\theta|_{\P^4}=-\frac{7}{2}h^2$.
 \end{lemma}
 \begin{proof}
 Let $\theta|_{\P^4}=nh^2$.  Equation \eqref{reln} implies that $60n^2=7\cdot12n(-5)-31(-5)^2+4(10)$ which implies the lemma.
\end{proof}
Finally, the last intersection theoretic piece of data we need
is
\begin{equation}\label{self}
[\P^4]^2=\chern{4}(N_{\P^4/X})=\chern{4}(T_{\P^4}^*)=5
\end{equation}
since $\P^4$ is Lagrangian.

Assume now that $X$ is deformation equivalent to a Hilbert
scheme of 4 points on a $K3$ surface.  Let $\ell\in H_2(X,\Z)$ be the class of the line, and $\lambda=6\ell\in H^2(X,\Z)$, via the embedding $H^2(X,\Z)\subset H_2(X,\Z)$ induced by the Beauville-Bogomolov form.  Note that $\lambda|_{\P^4}= \frac{(\lambda,\lambda)}{6} h$ since $\langle \lambda|_{\P^4},\ell\rangle=\langle \lambda,\ell\rangle=\frac{1}{6}(\lambda,\lambda)$ by the definition of $\lambda$.  Then
\[[\P^4]=a\lambda^4+b\lambda^2\theta+c\lambda^2\chern{2}(X)+d\theta^2+e\theta\chern{2}(X)+f\chern{2}(X)^2+g\theta\alpha+h\chern{2}(X)\alpha\]
Assume that $\alpha|_{\P^4}=yh^2$, for $y\in\Q$.  Intersecting this class with each of \eqref{basis}, 
\[\lambda^4,\lambda^2\theta,\lambda^2\chern{2}(X),\theta^2,\theta
\chern{2}(X),\chern{2}(X)^2,\alpha\theta,\alpha\chern{2}(X)\] yields by Lemmas \ref{restrict} and \ref{thetares} the equation
\begin{equation}
M(\lambda)[\P^4]=\begin{pmatrix}
\left(\frac{(\lambda,\lambda)}{6}\right)^4\\
-\frac{7}{2}\left(\frac{(\lambda,\lambda)}{6}\right)^2\\
-5\left(\frac{(\lambda,\lambda)}{6}\right)^2\\
\frac{49}{4}\\
\frac{35}{2}\\
25\\
-\frac{7}{2}y\\
-5y
\end{pmatrix}
\end{equation}
from which it follows that
\begin{equation}[\P^4]=
\begin{pmatrix}
\frac{1}{608256}\left(25+\frac{700}{(\lambda,\lambda)}+\frac{1764}{(\lambda,\lambda)^2}\right)\\
-\frac{1}{2737152}\left(25(\lambda,\lambda)+3276+\frac{15876}{(\lambda,\lambda)}\right)\\
\frac{1}{38016}\left(23+\frac{126}{(\lambda,\lambda)}\right)\\
\frac{1}{5474304}\left((\lambda,\lambda)^2+252(\lambda,\lambda)-41148\right)\\
-\frac{1}{190080}\left(5(\lambda,\lambda)-2142\right)\\
-\frac{1}{240}\\
\frac{31y}{1188}\\
-\frac{7y}{396}
\end{pmatrix}
\label{vector}\end{equation}
Finally, \eqref{self} yields:

\[5=\frac{25}{788299776}x^4 + \frac{175}{98537472}x^3 + \frac{403}{10948608}x^2 - \frac{7}{2376}y^2 +
\frac{7}{33792}x + \frac{65}{67584}\]
where $x=(\lambda,\lambda)$.  This may be rewritten as
\begin{equation}y^2=\frac{5^2}{2^{12}\cdot 3^4\cdot
  7}x^4+\frac{5^2}{2^9\cdot 3^4}x^3+\frac{13\cdot 31}{2^9\cdot
  3^2\cdot 7}x^2+\frac{3^2}{2^7}x-\frac{3^2\cdot 5\cdot 7^2\cdot
  197}{2^8}\label{dios}\end{equation}
Note that while we may have $y\in\Q$, $x$ must be integral.  Also note that there is a solution compatible with Conjecture \ref{conj}, namely $(x,y)=(-126,0)$.  By the analysis of the next section,
\begin{proposition}The only solution of \eqref{dios} with $x\in \Z$ and $y\in\Q$ is $(x,y)=(-126,0)$.
\end{proposition}
It then follows that
 \begin{theorem}\label{mainthm}Let $X$ be of $K3^{[4]}$-type, $\P^4\subset X$ be a smoothly embedded Lagrangian 4-plane, $\ell\in H_2(X,\Z)$ the class of a line in $\P^4$, and $\rho=2\ell\in H^2(X,\Q)$.  Then $\rho$ is integral, and
 \begin{equation}[\P^4]=\frac{1}{337920}\left(880\rho^4+1760\rho^2\chern{2}(X)-3520\theta^2+4928\theta
 \chern{2}(X)-1408\chern{2}(X)^2\right)\label{p4}\end{equation}
 Further, we must have $(\ell,\ell)=-\frac{7}{2}$.
 \end{theorem}
 \begin{proof}\eqref{p4} is obtained from \eqref{vector} by
 substituting $(\lambda,\lambda)=-126$ and $y=0$, after setting $\rho=\frac{1}{3}\lambda$.  It remains to
 show that $\rho$ is integral.  Following \cite{HHT}, after deforming to a Hilbert scheme of points on a $K3$ surface $S$, we
 can write
 \[\ell=D+m\delta^\vee\]
 using the decomposition dual to \eqref{decomp}, for $D\in H_2(S,\Z)$.  Since
 \[(\ell,\ell)=D^2-\frac{m^2}{6}=-\frac{7}{2}\]
 and $D^2\in 2\Z$, $3|m$.  For $2\ell$ to be an integral class in $H^2(X,\Z)$, by Poincar\'{e} duality it is sufficient for the form $(2\ell,\cdot)$ on $H_2(X,\Z)$ to be integral, which it obviously is, since $(\delta^\vee,\delta^\vee)=-\frac{1}{6}$.
 \end{proof}
\section{Solving the Diophantine equation}

The Diophantine equation \eqref{dios} to solve is
\[y^2=\frac{5^2}{2^{12}\cdot 3^{4}\cdot 7} x^4+\frac{5^2}{2^{9}\cdot 3^{4}} x^3+\frac {13\cdot 31}{2^{9} \cdot3^{2}\cdot 7} x^2+\frac{3^2}{2^7}x-\frac{3^2\cdot 5\cdot 7^2\cdot 197}{2^8}\]
with $x\in \mathbb{Z}$ and $y\in \mathbb{Q}$. Let $\mathcal{C}$ be the affine curve described by the equation. After the change of variables $(x_1,y_1)=(x+126,2^6\cdot 3^2\cdot 7y)$, every point $(x,y)\in \mathcal{C}$ with $x\in \mathbb{Z}$ gives an integral point $(x_1, y_1)$ on the curve $\mathcal{C}_1$:
\[y_1^2=(5^2\cdot 7) x_1^4 - (2^6 \cdot 5^2\cdot 7^2) x_1^3 +
(2^7 \cdot 3^2 \cdot 7 \cdot 23 \cdot 71) x_1^2-(2^{11} \cdot
3^4 \cdot 7^2 \cdot 11^2) x_1\]

\begin{lemma}
For an integer $v$ consider the elliptic curve
$\mathcal{E}_v$ given by the Weierstrass equation
\[y_2^2=x_2^3 - (2^6\cdot 5^2\cdot 7^2\cdot v) x_2^2+(2^7\cdot 3^2\cdot 5^2\cdot 7^2\cdot 23\cdot 71\cdot v^2) x_2-(2^{11}\cdot 3^4\cdot 5^4\cdot 7^4\cdot 11^2\cdot v^3)\]
Then every integral point $(x_1,y_1)\neq (0,0)$ on the curve
$\mathcal{C}_1$ corresponds to an integral point $(x_2,y_2)$ on one of
the curves $\mathcal{E}_v$ where
\begin{align*}
x_1&=u^2v&x_2&= 5^2\cdot 7\cdot v^2u^2\\
y_1&=uvw&y_2&=5^2\cdot 7\cdot v^2w
\end{align*}
for some integers $u,v,w$ where $v$ is a divisor of $2\cdot
3\cdot 7\cdot 11$.
\end{lemma}
\begin{proof}
Certainly if $x_1=0$ then $y_1=0$ and it can be checked that if $y_1=0$ then $x_1=0$ is the only rational solution. So let us assume for the remaining that $x_1,y_1\neq 0$. Note that since $x_1\in \mathbb{Z}$ it follows
that $y_1\in \mathbb{Z}$ and $x_1\mid y_1^2$. Since $x_1,y_1\neq 0$ we may write $x_1=u^2v$ and
$y_1=uvw$ for $u,v,w\in \mathbb{Z}$ with $v$
square-free. Rewriting the equation we get
\[vw^2 = 5^2\cdot 7\cdot u^6v^3 - 2^6 \cdot 5^2\cdot 7^2\cdot u^4v^2 +
2^7 \cdot 3^2 \cdot 7 \cdot 23 \cdot 71\cdot u^2v-2^{11} \cdot
3^4 \cdot 7^2 \cdot 11^2\]
and we conclude that $v$ is a divisor of $2 \cdot
3 \cdot 7 \cdot 11$.

Multiplying by $5^4\cdot 7^2\cdot v^3$ and making the change of
variables $y_2 = 5^2\cdot 7\cdot v^2\cdot w$ and $x_2=5^2\cdot
7\cdot v^2\cdot u^2$ we get the equation
\begin{align*}
(5^2\cdot 7\cdot v^2\cdot w)^2&=(5^2\cdot 7\cdot v^2\cdot u^2)^3
- 2^6\cdot 5^2\cdot 7^2\cdot v\cdot (5^2\cdot 7\cdot v^2\cdot
u^2)^2\\
&\quad+2^7\cdot 3^2\cdot 5^2\cdot 7^2\cdot 23\cdot 71\cdot v^2
(5^2\cdot 7\cdot v^2\cdot u^2)-2^{11}\cdot 3^4\cdot 5^4\cdot 7^4\cdot 11^2\cdot v^3
\end{align*}
which yields
\begin{align*}
y_2^2&=x_2^3 - 2^6\cdot 5^2\cdot 7^2\cdot v\cdot x_2^2+2^7\cdot 3^2\cdot 5^2\cdot 7^2\cdot 23\cdot 71\cdot v^2 x_2-2^{11}\cdot 3^4\cdot 5^4\cdot 7^4\cdot 11^2\cdot v^3
\end{align*}
and therefore a point $(x_2,y_2)\in \mathcal{E}_v(\mathbb{Z})$.
\end{proof}

Thus to find the required points on $\mathcal{C}$ we need to find the integral solutions of the elliptic curve $\mathcal{E}_v$ above whenever $v$ is a divisor of $2 \cdot
3 \cdot 7 \cdot 11$, of which there are $32$ (positive and
negative).

\begin{lemma}\label{lemma1} Suppose
$v$ is a divisor of $2 \cdot
3\cdot 7 \cdot 11$ such that $7\nmid v$.  If the curve $\mathcal{E}_v$ has an integral solution $(5^2\cdot 7\cdot u^2v^2,5^2\cdot 7\cdot v^2w)$ then $v\in \{-1,-2,-11,-22\}$.
\end{lemma}
\begin{proof}
Note from the equation
\[vw^2 = 5^2\cdot 7\cdot u^6v^3 - 2^6 \cdot 5^2\cdot 7^2\cdot u^4v^2 +
2^7 \cdot 3^2 \cdot 7 \cdot 23 \cdot 71\cdot u^2v-2^{11} \cdot
3^4 \cdot 7^2 \cdot 11^2\]
we deduce that $7\mid vw^2$. Since $7\nmid v$ it follows that $7\mid w$ so it must be that $5^2u^6v^3+2^7\cdot
3^2\cdot 23\cdot 71u^2v\equiv 0\pmod{7}$ in other words
$u^2v\equiv 3u^6v^3\pmod{7}$. Since $v$ is invertible we get
$5u^2\equiv u^6v^2$. If $7\nmid u$ then we would have that $5$
is a quadratic residue mod 7, which is not true. So $7\mid
u$. Rewriting the equation for $w=7w_1$ and $u=7u_1$ we get
\[vw_1^2=5^2\cdot 7^5\cdot u_1^6v^3-2^6\cdot 5^2\cdot 7^4\cdot
u_1^4v^2+2^7\cdot 3^2\cdot 7\cdot 23\cdot 71\cdot u_1^2v-2^{11}\cdot
3^4\cdot 11^2\]so necessarily $vw_1^2\equiv 3\pmod{7}$. But the
only square-free divisors $v$ of $2\cdot 3\cdot 11$ for which
such $w_1$ exist are $3,6,33,66,-1,-2,-11,-22$.

If $3\mid v$ then we could write $v=3v_1$ so we would get
\[v_1w_1^2=5^2\cdot 3^2\cdot 7^5\cdot u_1^6v_1^3-2^6\cdot 3\cdot 5^2\cdot 7^4\cdot
u_1^4v_1^2+2^7\cdot 3^2\cdot 7\cdot 23\cdot 71\cdot u_1^2v_1-2^{11}\cdot
3^3\cdot 11^2\]
which would imply that $3\mid v_1w_1^2$. Since $3\nmid v_1$ (as $v$ is square-free) it follows that $3^2\mid v_1w_1^2$ but then $3^2$ divides the right hand side so we deduce that $3\mid u_1$. Writing $w_1=3w_2$ and $u_1=3u_2$ we get
\[v_1w_2^2=5^2\cdot 3^6\cdot 7^5\cdot u_2^6v_1^3-2^6\cdot 3^3\cdot 5^2\cdot 7^4\cdot
u_2^4v_1^2+2^7\cdot 3^2\cdot 7\cdot 23\cdot 71\cdot u_2^2v_1-2^{11}\cdot
3\cdot 11^2\]
As before, we get that $3^2\mid v_1w_2^2$ but now $3^2$ cannot divide
the right hand side.

The remaining possibilities for $v$ are $-1,-2,-11,-22$.
\end{proof}

\begin{lemma}\label{lemma2}
If the curve $\mathcal{E}_v$ where $v$ is a divisor of $2 \cdot
3 \cdot 7 \cdot 11$ such that $7\mid v$ has an integral solution $(5^2\cdot 7\cdot u^2v^2,5^2\cdot 7\cdot v^2w)$ then $v\in \{7,14,77,154\}$.
\end{lemma}
\begin{proof}
Writing $v=7v_1$ we get
\[v_1w^2=5^2\cdot 7^3\cdot u^6v_1^3 - 2^6 \cdot 5^2\cdot 7^3\cdot u^4v_1^2 +
2^7 \cdot 3^2 \cdot 7 \cdot 23 \cdot 71\cdot u^2v_1-2^{11} \cdot
3^4 \cdot 7 \cdot 11^2\]
Since $v$ is square-free $7\nmid v_1$ so we deduce that
$7\mid w$. Writing $w=7w_1$ we get
\[7v_1w_1^2=5^2\cdot 7^2\cdot u^6v_1^3 - 2^6 \cdot 5^2\cdot 7^2\cdot u^4v_1^2 +
2^7 \cdot 3^2 \cdot 23 \cdot 71\cdot u^2v_1-2^{11} \cdot
3^4 \cdot 11^2\]
which implies that $u^2v_1\equiv 4\pmod{7}$. The only $v_1$ among
the square-free divisors of $2\cdot 3\cdot 11$ for which such
$u$ exist are $1,2,11,22,-3,-6,-33,-66$ giving $v\in \{7, 14, 77,
154,-21,-42,-231,-462\}$.

As in the previous lemma, under the assumption that $3\mid v$ we get a contradiction. The remaining possibilities are $v\in \{7, 14, 77,154\}$.
\end{proof}
Six of the eight cases to which we've reduced in Lemmas \ref{lemma1} and \ref{lemma2} are then treated directly by:
\begin{lemma}\label{lemma3}
If $v\in \{-1,-2,7,14,77,154\}$ the curve $\mathcal{E}_v$ has no integral points of the form $(5^2\cdot 7\cdot u^2v^2,5^2\cdot 7\cdot v^2w)$.
\end{lemma}
\begin{proof}
We will compute the integral points of these elliptic curves using Sage (\cite{sage}) version 5.2 run on William Stein's cluster \texttt{geom.math.washington.edu}. The general method is by finding a basis for the Mordell-Weil group of a rational elliptic curve (using the command \texttt{gens} in Sage) and then finding a list of all the integral points using this basis (using the command \texttt{integral\_points(mw\_basis=\ldots)} in Sage). Typically the computation of a basis is very difficult computationally (on the order of hours for the curves under consideration), whereas the computation of integral points is quite fast (on the order of seconds). As such we include bases for the Mordell-Weil groups of these elliptic curves in which case the computation of integral points can be reproduced quickly.

Using Sage we find that the curves $\mathcal{E}_7$ (rank 1 with generator $(\frac{23929444}{81}, \frac{22042862072}{729})$ found in 22 seconds), $\mathcal{E}_{77}$ (rank 1 with generator $(\frac{142777144885734591204}{47183614355089}, \frac{51150220299670713464643520008}{324105804064380058937})$ found in 6 hours 45 minutes) and $\mathcal{E}_{154}$ (rank 1 with generator $(\frac{267909856900}{23409}, -\frac{74537431985630600}{3581577})$ found in 6 hours 40 minutes) have no integral points.  Further computations show that the curves $\mathcal{E}_{14}$ (rank 4 with generators $(564480, 49392000)$, $(940800, 451113600)$, $(1317120, 945033600)$ and $(2257920, 2617776000)$ found in 15 seconds), $\mathcal{E}_{-1}$ (rank 3 with generators $(-27900, 2266200)$, $(138825/4, 125561925/8)$ and $(166980 : 85186200)$ found in 1 hour 20 minutes) and $\mathcal{E}_{-2}$ (rank 2 with generators $(-\frac{40566784}{529}, \frac{27776430464}{12167})$ and $(-\frac{3296728575}{65536}, -\frac{114720819732225}{16777216})$ found in 1 hour 18 minutes) have integral points, but none of them has the $x$-coordinate of the required form. Indeed, $\mathcal{E}_{-1}$ has 6 integral points, $(-39196,\pm156792)$, $(-27900,\pm2266200)$ and $(166980,\pm85186200)$, but none of the $x$-coordinates are of the required form $x=5^2\cdot 7\cdot (-1)^2\cdot u^2$; $\mathcal{E}_{-2}$ has two integral points $(0,\pm15523200)$ but the $x$-coordinate was assumed to be nonzero; finally, $\mathcal{E}_{14}$ has $34$ integral points $(564480,\pm49392000)$, $(604905,\pm101433675)$, $(632100,\pm129859800)$, $(683844,\pm180931128)$, $(755825,\pm251976375)$, $(940800,\pm451113600)$, $(1063680,\pm599510400)$, $(1317120,\pm945033600)$, \allowbreak
$(1361220,\pm1010272200)$, $(2257920,\pm2617776000)$,
$(3066624,\pm4451914368)$, $(3327780,\pm5110549800)$, 
$(11863929,\pm38995732083)$, $(12603780,\pm42818542200)$,
$(13848576,\pm49513570176)$, $(72195620,\pm608777597400)$ and
$(1964277504,\pm87032792472192)$ but none of the $x$-coordinates
are of the form $2^2\cdot 5^2\cdot 7^3\cdot u^2$.
\end{proof}
The remaining two curves $\mathcal{E}_{-11},\mathcal{E}_{-22}$ are computationally less tractable.  The standard computation of generators for the Mordell-Weil group in Sage for these two elliptic curves does not terminate in any reasonable time, though the closed-source algebra system Magma (\cite{magma}) allows one to perform a reasonably fast analysis of these two elliptic curves. We will give two computational proofs that these curves do not have integral points of the required type: the first, in the open source Sage, relies on Kolyvagin's proof of the Birch and Swinnerton-Dyer conjecture of elliptic curves over $\mathbb{Q}$ of analytic rank 1 while the second, in the proprietary Magma, uses a two descent procedure, and is given mainly as a corroboration of the results from Sage. We are greatful to Michael Stoll for explaining how to do
the computations in Magma.  We remark that the same methods will in principle work for the other curves in Lemma \ref{lemma3} of rank 1, namely $\mathcal{E}_{77}$ and $\mathcal{E}_{154}$.

We first need the following lemma.
\begin{lemma}\label{lemma:e1122}
If $E$ is one of the curves $\mathcal{E}_{-11}$ and $\mathcal{E}_{-22}$ then $L'(E,1)\neq 0$.
\end{lemma}
\begin{proof}
We recall a result of Cohen (\cite[5.6.12]{cohen:computational})
that
\[L'(E,1) = 2\sum_{n\geq 1}\frac{a_n}{n}E_1\left(\frac{2\pi
  n}{\sqrt{N}}\right)\]
where $N$ is the conductor of $E$ and $E_1(x)=\int_1^\infty e^{-xy}y^{-1}dy$ is the exponential integral. Truncating this series at $k$, one gets $L'(E,1) = L_k+\varepsilon_k$ where $\displaystyle L_k=2\sum_{n=1}^k\frac{a_n}{n}E_1\left(\frac{2\pi
  n}{\sqrt{N}}\right)$ and the error is explicitly bounded
$|\varepsilon_k|\leq
2e^{-2\pi(k+1)/\sqrt{N}}/(1-e^{-2\pi/\sqrt{N}})$ (for a proof see
\cite[\S 2.2]{andrei}). This estimate is at the basis of the
Sage command \texttt{E.lseries().deriv\_at1(k)} (here $k$ is the
cutoff). In principle, if one expects that $L'(E,1)\neq 0$ then
it suffices to choose the cutoff index $k$ large enough that
$|\varepsilon_k|<|L_k|$ in which case $L'(E,1)$ will be forced to
be nonzero.

However, the curves under consideration have such a large
conductor (in both cases $N=83060209520534400$) that $k$ has to
be choosen on the order of $8\cdot 10^8$, which is too large for
practical purposes in Sage: in effect one runs out of memory in
the computation of the coefficients $a_n$ and $E_1(2\pi
n/\sqrt{N})$. We compute the coefficients $a_n$ for the two
curves up to $k=8\cdot 10^8$ by first computing $a_p$ for $p$
prime (this operation takes about 2 hours for each curve) and
then reconstructing $a_n$ using the following: if $(m,n)=1$ then
$a_{mn}=a_ma_n$, if $p\nmid N$ then
$a_{p^k}=a_pa_{p^{k-1}}-pa_{p^{k-2}}$ and if $p\mid N$ then
$a_{p^k}=a_p^k$. For each curve the resulting file is on the
order of $2.5GB$ and the computation takes about 3 hours for
each curve. Next, we compute $E_1(2\pi n/\sqrt{N})$ for $1\leq
n\leq 8\cdot 10^8$ (once, as the two curves have the same
conductor). The command
\texttt{exponential\_integral\_1($2\pi/\sqrt{N}$, k)} in Sage
should return the desired list but $k$ is too large for this
operation to be feasible. Instead, noting that Sage's
\texttt{exponential\_integral\_1} is a wrapper for the PARI (\cite{PARI2}, version 2.5.4)
function \texttt{veceint1}, we rewrote this PARI function to
write the coefficients $E_1(2\pi n/\sqrt{N})$ to a file, instead
of collecting them in a prohibitively long vector. The
subsequent computation was run for about 10 hours resulting in
35GB of data.

Each coefficient $E_1(2\pi n/\sqrt{N})=E_{1, n}+\varepsilon_{1,
  n}$ where $E_{1,n}$ is the number computed in PARI and
$|\varepsilon_{1,n}|<10^{-20}$ is the chosen precision. We denote by $\ell_E$ the value $2\sum_{n=1}^k \frac{a_n}{n}E_{1,n}$ computed in Sage and PARI using the cutoff $k=8\cdot 10^8$. Therefore we compute the value
of $L'(E, 1)=\ell_E+\varepsilon$ where the error
is then at most (using the inquality $|a_n|\leq n$ from \cite[Lemma 2.9]{andrei})
\begin{align*}
\varepsilon&<2\sum_{n=1}^k \frac{|a_n|}{n}\cdot
10^{-20}+\varepsilon_k\\
&<2\cdot 10^{-20}\cdot k + \varepsilon_k\\
&<16\cdot 10^{-12}+\varepsilon_k\\
&<3
\end{align*}
Finally, in Sage we find $\ell_{\mathcal{E}_{-11}}=12.561\real{3613056383}$ and $\ell_{\mathcal{E}_{-22}}=16.069\real{3858103800}$ and the conclusion follows.

\end{proof}

\begin{lemma}
If $v\in \{-11,-22\}$ the curve $\mathcal{E}_v$ has no integral points of the form $(5^2\cdot 7\cdot u^2v^2,5^2\cdot 7\cdot v^2w)$.
\end{lemma}
\begin{proof}

First, suppose $E/\mathbb{Q}$ is an elliptic curve of rank 1 and $P\in E(\mathbb{Q})$ is a point of infinite order (a fact which can be checked computationally by requiring that the canonical height of the point is nonzero). We would like a fast algorithm for finding a generator $P_0$ of the Mordell-Weil group $E(\mathbb{Q})$. Suppose $P_0$ is a generator of $E(\mathbb{Q})$ in which case $P=nP_0$ for some integer $n$ as $E$ has rank 1. If $P$ is not a generator then $|n|\geq 2$.

Write $h$ for the logarithmic height and $\widehat{h}$ for the canonical logarithmic height on $E(\mathbb{Q})$. There exists a constant $B$, depending only on $E$, called the Cremona-Pricket-Siksek bound, such that for all $Q\in E(\mathbb{Q})$, $h(Q)\leq \widehat{h}(Q)+B$. Given a Weierstrass equation for $E$, the constant $B$ can be computed in Sage using the command \texttt{CPS\_height\_bound} and in Magma using the command \texttt{SiksekBound}. If $|n|\geq 2$ then $\widehat{h}(P_0)\leq\dfrac{\widehat{h}(P)}{n^2}\leq
\dfrac{\widehat{h}(P)}{4}$ so
$h(P_0)=\widehat{h}(P_0)+h(P_0)-\widehat{h}(P_0)\leq
\dfrac{\widehat{h}(P)}{4}+B$. Thus, to find $P_0$ one only needs
to search for rational points of height at most
$\dfrac{1}{4}\widehat{h}(P)+B$. One can find rational points of
height $\leq h_0$ in Sage using the command
\texttt{rational\_points(bound=$h_0$)} and a generator $P_0$ can
be found in the resulting finite list.

We will first check that the elliptic curves $E_{-11}$ and
$E_{-22}$ have rank 1 and then we will apply the above described
procedure to find a basis for the Mordell-Weil group.  The command
\texttt{DescentInformation} in Magma rapidly returns rank 1 for
our curves. As mentioned above, in Sage one needs a different approach (note that
the Sage command \texttt{analytic\_rank} yields only the {\it
  probable} analytic rank, equal to 1, in about 17 hours for
each of the two curves). 

Recall Kolyvagin's result that if $E$ is a (necessarily modular) rational elliptic curve of analytic rank 0 or 1 then the Birch and Swinnerton-Dyer conjecture is true, i.e., the rank of the elliptic curve equals its analytic rank. We will exhibit below points of infinite order on each of the two elliptic curves and so their rank (and so also their analytic rank) is at least 1. Lemma \ref{lemma:e1122} implies that $L'(E,1)\neq 0$ and so their analytic rank, and therefore also their rank, must be 1, as desired.

We proceed with finding bases for the Mordell-Weil groups. We start with the curve $E=\mathcal{E}_{-11}$.
The elliptic curve $E$ is
\[y^2 = x^3 + 2^6\cdot 5^2\cdot 7^2\cdot 11\cdot x^2 + 2^7\cdot 3^2\cdot 5^2\cdot 7^2\cdot 11^2\cdot 23\cdot 71\cdot x +2^{11}\cdot 3^4\cdot 5^4\cdot 7^4\cdot 11^5\]
Via the change of variables $x=4x_1-287468,y=8y_1$ we get the
minimal Weierstrass equation $E'$
\[y_1^2 = x_1^3 - x_1^2  + 1933249267x_1 + 116312127942837\]
One may easily check that the point \[P=\left(\dfrac{195693}{4},\dfrac{144883425}{8}\right)\] is in $E'(\mathbb{Q})$ (this point was found using Magma, but checking that it is a point on the curve is immediate without necessarily using a computer). The command \texttt{height} in Sage computes the canonical height to be $\widehat{h}(P)=11.289\real{8464758516299021967038460}$ (and so $P$ has infinite order) while the CPS bound is $B=11.424\real{0032533635864850294780554}$. 

As explained before, we seek a generator of $E'(\mathbb{Q})$. If $P$ is not a generator then a generator will have height at most $\widehat{h}(P)/4+B$. However, a computation in Sage shows that the only rational points with this height bound are $0, \pm P$ and so $P$ must be a generator of $E'(\mathbb{Q})$.

Transfering back to $E(\mathbb{Q})$ one obtains the generator
$(x,y)=(-91775,144883425)$ of $E(\mathbb{Q})$. Using the command \texttt{integral\_points} in Sage to
compute the integral points, inputting manually the basis for
$E(\mathbb{Q})$, one obtains that $E(\mathbb{Q})$ has the integral points $(-91775,\pm 144883425)$ but
$x=-91775$ is not of the required form.

The elliptic curve $E=\mathcal{E}_{-22}$ is 
\[y^2 = x^3 + 2^7\cdot 5^2\cdot 7^2\cdot 11\cdot x^2 + 2^9\cdot
3^2\cdot 5^2\cdot 7^2\cdot 11^2\cdot 23\cdot 71\cdot x
+2^{14}\cdot 3^4\cdot 5^4\cdot 7^4\cdot 11^5\]
via the change of variables $x=16x_1-574928,y=64y_1$ gives the minimal model $E'$
\[y_1^2 = x_1^3 + x_1^2 + 483312317x_1 + 14539257649013\]
Again one may easily check that the point $P=\left(-17428,-907137\right)$
is in $E'(\mathbb{Q})$. It has canonical height $\widehat{h}(P)=5.106\real{29161302709560617057630633}$ and thus it has infinite order. The CPS bound is computed to be $B=10.774\real{4647353275361503068823262}$. As before this allows one to show that $P$ is a generator of $E'(\mathbb{Q})$. The point $P$ corresponds to the point $(-853776, 58056768)$, a generator of $E(\mathbb{Q})$. Finally, using this basis in the computation of integral points in Sage yields that the only integral points are $(-853776,\pm 58056768)$ but $x$ cannot be $-853776$, which is negative, and hence not of the required form.
\end{proof}

\section{Appendix:  Equivariant Localization}

For the sake of completeness we describe the well-known computation of the
integrals
\[\int_{S^{[n]}}\delta^{k}\chern{\mu}(S^{[n]})\]
for $S=\P^2,\P^1\times\P^1$ and $\delta=\det\O^{[n]}$ by toric localization.
\newcommand{\GG}{\mathbb{G}}

First consider $S=\mathbb{A}^2$, which has an action by
$G=\GG_m^2$ via $(x,y)\mapsto(\alpha x,
\beta y)$ where $\alpha,\beta$ are the characters obtained by projecting to each factor.  The only fixed point is the origin $(0,0)$.  $G$ also
acts on $(\mathbb{A}^2)^{[n]}$; fixed points are length $n$
subschemes $Z$ fixed by $G$.  Thus, they must be supported on a
fixed point (i.e. the origin), and the ideal $I_Z\subset
A=\mathbb{C}[x,y]$ must be generated by monomials.  $I_Z$ is determined by the monomials $x^ay^b$ left out of the ideal, which form a Young tableau with
$n$ boxes.  Given such a Young tableau in the upper right quadrant, let $(i,b_i-1)$ for $0\leq i\leq n-1$ be the extremal boxes, so $b_i$ is the height of the $i$th column.  A partition $\mu$ of $n$ uniquely determines a Young tableau by arranging $\mu_i$ columns of height $i$ in descending order. 

For a space $X$ with an action by $G$ with isolated
fixed points, we can compute integrals over $X$ by restricting to the fixed point locus using Bott localization: 
\[\int_X \phi=\sum_{p\in X^G}\int
\frac{i_p^*\phi}{\chern{\mathrm{top}}(T_pX)}\]
where $\phi\in H^*_G(X)$, $i_p^*:H_G^*(X)\rightarrow H_G^*(X^G)\cong H^*(X^G)\otimes
H_G^*(\pt)$ is the pull-back to a fixed point $p\in X^G$.  The Chern class is the equivariant Chern class of the $G$
representation $T_pX$.

For example, consider $S=\mathbb{A}^2$ again.  For a partition $\mu$ representing
a fixed point $p_\mu$ of $X=(\mathbb{A}^2)^{[n]}$, the Chern polynomial is \cite[Lemma 3.2]{ellingsrud}
\begin{equation}C(\mu;\alpha,\beta):=\sum_i t^i\chern{2n-i}(T_{p_\mu}X)=\prod_{1\leq i\leq j\leq
  n}\prod_{s=b_j}^{b_{j-1}-1}(t+(i-j-1)\alpha+(b_{i-1}-s-1)\beta)(t+(j-i)\alpha+(s-b_{i-1})\beta)\end{equation}
$\O^{[n]}$ restricted to a point of $\mathbb{A}^{[n]}$ corresponding to a subscheme $Z$ is canonically $\O_Z$, so setting $f=\chern{1}(\O^{[n]})$,
\begin{equation}Z(\mu;\alpha,\beta):=i_{p_\mu}^*f=\sum_{i=0}^n\sum_{j=0}^{b_i-1}(i\alpha+j\beta)\end{equation}

\subsection{The case $S=\mathbb{P}^2$}  Let $\GG_m^2$ act on $[x,y,z]$ via $[\alpha x, \beta y, z]$.  There are three fixed points $p_0=[0,0,1], p_1=[0,1,0]$, $p_2=[1,0,0]$, and a length $n$ subscheme $Z$ of $\P^2$ will consist
of a length $n_i$ subscheme $Z_i$ at $p_i$ with $\sum n_i=n$.  The
tangent space at such a point is canonically
\[T_Z(\P^2)^{[n]}=\bigoplus_i T_{Z_i}(\P^2)^{[n_i]}\]
Note that at any point $[Z]\in(\P^2)^{[n]}$ corresponding to a subscheme $Z$ supported at $p_i$, there is a $\GG_m^2$-stable Zariski neighborhood isomorphic to $\mathbb{A}^{[n]}$ with torus action via $(\alpha x,\beta y)$, $(\alpha\beta^{-1}x,\beta^{-1}y)$, $(\beta\alpha^{-1}x,\alpha^{-1}y)$ for $i=0,1,2$ respectively. 
A 3-vector
partition $\underline\mu$ of $n$ will be three partitions
$(\mu_1,\mu_2,\mu_3)$ such that $|\mu_1|+|\mu_2|+|\mu_3|=n$;
3-vector partitions of $n$ classify fixed points $p_{\underline\mu}$ of
$X=(\P^2)^{[n]}$.  By the above, the tangent space at $p_{\underline{\mu}}$ has
Chern polynomial
\begin{equation}\sum
t^iC_{2n-i}(\underline\mu;\alpha,\beta)=C(\mu_1;\alpha,\beta)C(\mu_2;\alpha-\beta,-\beta)C(\mu_3;\beta-\alpha,-\alpha)\end{equation}
Define
$C_i(\underline\mu;\alpha,\beta)=\chern{i}(T_{p_{\underline{\mu}}}X)$.
Also,
\begin{equation}Z(\underline\mu;\alpha,\beta):=i_{p_{\underline\mu}}^*f=Z(\mu_1;\alpha,\beta)+Z(\mu_2;\alpha-\beta,-\beta)+Z(\mu_3;\beta-\alpha,-\alpha)\end{equation}
The final answer is then, for $X=(\P^2)^{[n]}$
\begin{equation}\int_{X}f^{2n-\sum_i k_i}\prod_i\chern{k_i}(TX)=\sum_{\underline{\mu},
  |\underline{\mu}|=n}\frac{Z(\underline{\mu};\alpha,\beta)^{2n-\sum_i k_i}\prod_i C_{k_i}(\underline\mu;\alpha,\beta)}{C_{2n}(\underline{\mu};\alpha,\beta)}\end{equation}

\subsection{The case $S=\P^1\times\mathbb{P}^1$}Let $\GG_m^2$ act on $S=\P^1\times\P^1$ via $[\alpha x_1,y_1]\times[\beta x_2,y_2]$.  The fixed points  are classified by 4-vector partitions $\underline\mu$.  Now we have
\begin{equation}\sum
t^iC_{2n-i}'(\underline\mu;\alpha,\beta)=C(\mu_1;\alpha,\beta)C(\mu_2;-\alpha,\beta)C(\mu_3;\alpha,-\beta)C(\mu_4;-\alpha,-\beta)\end{equation}
Also,
\begin{equation}Z'(\underline\mu;\alpha,\beta):=i_{p_{\underline\mu}}^*f=Z(\mu_1;\alpha,\beta)+Z(\mu_2;-\alpha,\beta)+Z(\mu_3;\alpha,-\beta)+Z(\mu_4;-\alpha,-\beta)\end{equation}
The final answer is then once again
\begin{equation}\int_{(\P^1\times\P^1)^{[n]}}f^{2n-\sum_i k_i}\prod_i\chern{k_i}(TX)=\sum_{\underline{\mu},
  |\underline{\mu}|=n}\frac{Z'(\underline{\mu};\alpha,\beta)^{2n-\sum_i k_i}\prod_iC'_{k_i}(\underline\mu;\alpha,\beta)}{C'_{2n}(\underline{\mu};\alpha,\beta)}\end{equation}

 \subsection{Universal Polynomials}  Let $\Phi$ be the universal genus from Section \ref{fujikisection}.  We have
\[\sum_{n\geq0}z^n\int_{S^{[n]}}\exp\det(\O^{[n]})\Phi(S^{[n]})=\mathbf{A}(z)^{\chern{1}(S)^2}\mathbf{B}(z)^{\chern{2}(S)}\]
We have computed explicitly in SAGE the power series $\mathbf{A}$ and $\mathbf{B}$ for vanishing odd Chern classes up to degree 20, and the result can be found on the authors' webpages. For illustration, we include the formula up to degree 2:
\begin{align*}
\Phi = 1+\chern{2}(a_1^2-2a_2)+\chern{2}^2(a_2^2-2a_1a_3+2a_4)+\chern{4}(a_1^4-4a_1^2a_2+2a_2^2+4a_1a_3-4a_4)+\cdots
\end{align*}

By localization, we compute:
\begin{align*}
\mathbf{A}(z)&=1+a_2z\\
&+z^2\left(- a_{1}^{3} + 3 a_{1}^{2} a_{2}+\frac{1}{4} a_{1}^{2} + a_{1} a_{2} - \frac{9}{2} a_{2}^{2} +
a_{1} a_{3} + \frac{1}{6} a_{1} - \frac{3}{2} a_{2} + 3 a_{3} -
10 a_{4} - \frac{1}{48}\right) + O(z^3)\\
\mathbf{B}(z)&=1+z\left(a_{1}^{2} - 2 a_{2}\right)\\
&+z^2\left(2a_{1}^{4} - 8 a_{1}^{2} a_{2} - \frac{5}{4} a_{1}^{2} +
\frac{31}{2} a_{2}^{2} - 15 a_{1} a_{3} + \frac{5}{2} a_{2} + 15
a_{4} + \frac{1}{48}\right)+O(z^3)\\
\end{align*}

\bibliography{biblio}
\bibliographystyle{alpha}
\end{document}